\documentclass[11pt,reqno]{amsart}
\usepackage[margin = 1.3 in]{geometry}
\usepackage{
  hyperref,
  amsmath,
  amssymb,
  amsthm,
  thmtools,
  microtype,
  mathrsfs,
  enumitem,
  stmaryrd,
  diagbox
}
\usepackage[table]{xcolor}
\usepackage{tikz, tikz-cd}

\usepackage{graphicx}

\usepackage{eucal}

\theoremstyle{plain}
\newtheorem{theorem}{Theorem}[section]
\newtheorem{proposition}[theorem]{Proposition}
\newtheorem{lemma}[theorem]{Lemma}
\newtheorem{conjecture}[theorem]{Conjecture} 
\newtheorem{corollary}[theorem]{Corollary}
\theoremstyle{definition}
\newtheorem{definition}[theorem]{Definition}

\newtheorem{example}[theorem]{Example}
\newtheorem{exercise}[theorem]{Exercise}

\newtheorem{question}[theorem]{Question}
\theoremstyle{remark}

\numberwithin{equation}{section}

\title[Projection and ramification]{Ramification divisors of general projections}
\author[Deopurkar]{Anand Deopurkar}
\address{Mathematical Sciences Institute\\Australian National University, Acton, ACT, Australia}
\email{anand.deopurkar@anu.edu.au}
\author[Duryev]{Eduard Duryev}
\address{Institute de Math\'ematiques de Jussieu, Paris, France}
\email{edwardduriev@gmail.com}
\author[Patel]{Anand Patel}
\address{Department of Mathematics\\Oklahoma State University, Stillwater, OK, USA}
\email{anand.patel@okstate.edu}

\renewcommand{\k}{\mathbb{K}}
\DeclareMathOperator{\id}{id}
\DeclareMathOperator{\Bl}{Bl}

\DeclareMathOperator{\Hilb}{Hilb}
\DeclareMathOperator{\sing}{Sing}

\DeclareMathOperator{\F}{\mathbf F}

\renewcommand{\to}{{\longrightarrow}}

\ifcsname theorem\endcsname{}\else\declaretheorem[parent=section]{theorem}\fi
\ifcsname corollary\endcsname{}\else\declaretheorem[sibling=theorem]{corollary}\fi
\ifcsname lemma\endcsname{}\else\declaretheorem[sibling=theorem]{lemma}\fi
\ifcsname proposition\endcsname{}\else\declaretheorem[sibling=theorem]{proposition}\fi
\ifcsname conjecture\endcsname{}\else\fi
\ifcsname problem\endcsname{}\else\fi
\ifcsname question\endcsname{}\else\fi
\ifcsname definition\endcsname{}\else\declaretheorem[sibling=theorem, style=definition]{definition}\fi
\ifcsname exercise\endcsname{}\else\fi
\ifcsname example\endcsname{}\else\fi
\ifcsname remark\endcsname{}\declaretheorem[sibling=theorem, style=remark]{remark}\fi

\providecommand {\Z}{{\bf Z}}

\providecommand {\R}{{\bf R}}

\renewcommand {\P}{{\bf P}}
\providecommand {\Gr}{{\bf Gr}}
\providecommand {\A}{{\bf A}}

\providecommand{\GL}{\operatorname{GL}}
\providecommand{\PGL}{\operatorname{PGL}}

\providecommand {\from}{{\colon}}

\providecommand{\spec}{\operatorname{Spec}}

\providecommand{\Hom}{\operatorname{Hom}}

\providecommand{\Aut}{\operatorname{Aut}}

\providecommand{\Pic}{\operatorname{Pic}}
\providecommand{\Sym}{\operatorname{Sym}}
\providecommand{\rk}{\operatorname{rk}} \declaretheorem[sibling=theorem,style=remark]{remark}
\numberwithin{equation}{section}
\declaretheorem[title=Theorem, style=plain]{maintheorem}

\renewcommand{\O}{\mathcal O}

\newcommand{\smvee}{\raise0.5ex\hbox{$\scriptscriptstyle\vee$}}
\DeclareMathOperator{\ord}{ord}

\newcommand{\Proj}{{\text{\rm Proj}\,}}

\begin{document}

\begin{abstract}
  We study the ramification divisors of projections of a smooth projective variety onto a linear subspace of the same dimension.
  We prove that the ramification divisors vary in a maximal dimensional family for a large class of varieties.
  Going further, we study the map that associates to a linear projection its ramification divisor.
  We show that this map is dominant for most (but not all!) varieties of minimal degree, using (linked) limit linear series of higher rank.
  We find the degree of this map in some cases, extending the classical appearance of Catalan numbers in the geometry of rational normal curves, and give a geometric explanation of its fibers in terms of torsion points of naturally occurring elliptic curves in the case of the Veronese surface and the quartic rational surface scroll.
\end{abstract}

\maketitle

\section{Introduction}\label{sec:intro}
Let $f \from X \to Y$ be a map between smooth algebraic varieties.
A fundamental object associated to $f$ is the set $R(f) \subset X$ consisting of critical points of $f$, namely the points $x \in X$ at which $df \from T_xX \to T_{f(x)}Y$ has less than maximal rank.
One might ask: To what extent can $f$ be recovered from $R(f)$?
For example, does every non-trivial perturbation of $f$ induce a non-trivial perturbation of $R(f)$?
If this is the case, then how many other $g \from X \to Y$ have $R(g) = R(f)$?
The goal of this paper is to better understand these questions when $X$ is a smooth projective variety and $Y$ is a projective space of the same dimension as $X$.

More precisely, let $X \subset \P^n$ be a smooth projective variety of dimension $r$, not contained in a hyperplane.
A general $(n-r-1)$-dimensional linear subspace $L \subset \P^n$ defines a finite surjective map $X \to \P^r$.
The set of critical points of this map is the ramification divisor $R(L) \subset X$.
By the Riemann--Hurwitz formula, $R(L)$ lies in the linear series $|K_X + (r+1)H|$, where $K_X$ is the canonical class, and $H$ is the hyperplane class on $X$.
The association $L \leadsto R(L)$ gives a rational map
\[
  \rho \from \Gr(n-r, n+1) \dashrightarrow |K_X + (r+1)H|.
\]
In terms of $\rho$, we can formulate questions more precisely.  Knowing the behavior of the ramification locus under a perturbation is equivalent to knowing whether the map $\rho$ is generically finite, or equivalently, whether the image of $\rho$ has the maximal possible dimension.
In the literature, this question is known as the question of maximal variation of the ramification locus.
Knowing the number of maps with the same ramification locus is knowing the degree of $\rho$.
Our main goal is to address these questions.

\subsection{Maximal variation}
Although very classical in nature, the question of maximal variation of $\rho$ first appeared explicitly in the work of Flenner and Manaresi \cite{fle.man:98} in connection with the transcendence degree of the St\"uckrad-Vogel cycle in intersection theory. 
They established maximal variation under a geometric condition we call ``incompressibility.''
We prove the following more general theorem.
\begin{maintheorem}\label{thm:main}
  Let $X \subset \P^n$ be a non-degenerate, normal, projective variety over a field of characteristic zero.
  Suppose at least one of the following holds:
  \begin{enumerate}
  \item\label{item:incomp}(incompressibility) for every linear subspace $L \subset \P^n$ of dimension $(n-r-1)$, projection from $L$ restricts to a dominant rational map $X \dashrightarrow \P^r$;
  \item\label{item:dual}(divisorial dual) the dual variety $X^* \subset {\P^n}^*$ is a hypersurface.
  \end{enumerate}
  Then $\rho$ is generically finite onto its image.
\end{maintheorem}
In the main text, \autoref{thm:main} is \autoref{cor:maintheorem}.

Recall that the dual variety $X^* \subset {\P^n}^*$ is the closure of the locus of hyperplanes $H \subset \P^n$ such that the intersection of $H$ with the smooth locus of $X$ is singular.
We call $X \subset \P^n$ satisfying \eqref{item:incomp} \emph{incompressible} as it cannot be projected down (compressed) to a smaller dimensional subvariety by a linear projection.
The maximal variation result of \cite{fle.man:98} obtained the conclusion of \autoref{thm:main} assuming incompressibility.

\autoref{thm:main} substantially increases the class of varieties where we now know maximal variation.
Indeed, it is easy to see that if $X$ is any smooth surface over a field of characteristic zero, then the dual variety $X^*$ is a hypersurface.
Therefore, maximal variation holds for all surfaces.
Note, in contrast, that not all surfaces are incompressible.
The first counterexample is the cubic surface scroll $X \subset \P^4$---the projection from the directrix $L \subset X$ projects $X$ onto a $\P^1$.
Thus, even for surfaces, condition \eqref{item:dual} of \autoref{thm:main} covers new ground.
In general, let $X$ be of arbitrary dimension embedded in $\P^n$ by a sufficiently positive line bundle (for example, by a sufficiently high Veronese re-embedding).
Then $X \subset \P^n$ is usually not incompressible, but the dual variety $X^*$ will be a hypersurface.
As a result, $X \subset \P^n$ is covered by condition \eqref{item:dual} of \autoref{thm:main}.

The hypotheses in \autoref{thm:main} are sufficient, but not necessary.
Indeed, consider $X = \P^{r-1} \times \P^1 \subset \P^{2r-1}$, embedded by the Segre embedding, for $r \geq 3$.
Then $X$ is neither incompressible nor is $X^*$ a hypersurface, and yet $\rho$ is dominant (see \autoref{thm:examples}).

Given that maximal variation holds for a large class of varieties, it is natural to wonder if it always holds.
This is not the case.
\begin{maintheorem}
  \label{thm:counterexamples}
  There exist smooth, non-degenerate, rational normal scrolls $X \subset \P^{n}$ of every dimension $r \geq 4$ and degree $d \geq r+1$ for which the projection-ramification map $\rho$ is not generically finite onto its image.
\end{maintheorem}
In the main text, \autoref{thm:counterexamples} is \autoref{cor:actualcounterexamples}.

The possible existence of varieties for which the projection-ramification map is not generically finite has been alluded to by Zak \cite{zak:}, and our examples of rational normal scrolls in \autoref{thm:counterexamples} are the first known instances.
We describe these rational normal scrolls  explicitly; it is worth mentioning that they include scrolls of general moduli.

Having considered the question of maximal variation in general, we turn our attention to cases where the map $\rho$ has a chance of being dominant.
Our next result classifies such $X \subset \P^n$.
\begin{maintheorem}\label{thm:minimaldegree}
  Let $X \subset \P^{n}$ be a smooth, non-degenerate projective variety of dimension $r$ over a field of characteristic zero.
  We have the inequality
  \[ \dim \Gr(n-r, n+1) \leq \dim |K_X + (r+1)H|.\]
  Equality holds if and only if $X$ is a variety of minimal degree, that is $\deg X = n-r+1$.
\end{maintheorem}
In the main text, \autoref{thm:minimaldegree} is \autoref{thm:actualminimaldegree}.

Recall the list of smooth varieties of minimal degree: quadric hypersurfaces, the Veronese surface in $\P^5$, and rational normal scrolls.
By \autoref{thm:main}, $\rho$ is dominant for hypersurfaces and surfaces, so what remains are the scrolls.
Among the scrolls, the curves (rational normal curves) and surfaces are again covered by \autoref{thm:main}.
For threefold scrolls, we show by an explicit calculation and a degeneration argument that $\rho$ is dominant (\autoref{cor:maxvariation3scrolls}).
In higher dimensions, the story is complicated, as evidenced by \autoref{thm:counterexamples}.
Nevertheless, we prove the following.
\begin{maintheorem}
  \label{thm:rationalnormalscrolls}
  Let $X = \P E \subset \P^n$ be a rational normal scroll, where $E$ is a ample vector bundle of rank $r$ on $\P^1$, general in its moduli.
  If $\deg E = a \cdot (r-1) + b \cdot (2r-1) + 1$ for non-negative integers $a, b$, then the projection-ramification map $\rho$ is dominant for $X$.
  In particular, the conclusion holds if $E$ is general of degree at least $(r-1)(2r-1) + 1$.
\end{maintheorem}
In the main text, \autoref{thm:rationalnormalscrolls} is \autoref{thm:actualrationalnormalscrolls}.

The proof of \autoref{thm:rationalnormalscrolls} goes by degeneration.
We degenerate $X$ to a reducible variety $X_0$, namely the projectivization of a vector bundle on a two-component nodal rational curve.
Suppose we could define a projection-ramification map for $X_0$ and show that it is dominant, then the same holds $X$, by the upper semi-continuity of fiber dimensions.
Although promising, this line of attack fails with the most na\"ive definition of the projection-ramification map.
The right definition requires more sophisticated tools, specifically, the spaces of (linked) limit linear series for vector bundles of higher rank developed by Teixidor i Bigas \cite{tei-i-big:91} and Osserman \cite{oss:14}.

\subsection{Enumerative problems}
\autoref{thm:minimaldegree} and \autoref{thm:rationalnormalscrolls} motivate a natural set of enumerative questions: for $X \subset \P^n$ of minimal degree, what is the degree of the projection-ramification map $\rho$?
We make the convention that if a map is not dominant, then its degree is $0$.

For $X$ of dimension 1, namely a rational normal curve, the answer is easy to find---the degree of $\rho$ is the Catalan number $\frac{(2n-2)!}{n!(n-1)!}$.
Indeed, in this case, the projection-ramification map
\[ \rho \from \Gr(2, n+1) \to \P^{2n-2}\]
is regular, and the pullback of $\O(1)$ is the Pl\"ucker line bundle.
Therefore, the degree of $\rho$ is the top self-intersection of the Pl\"ucker bundle.
Schubert calculus gives that this is the Catalan number.

For $X$ of codimension $1$, namely a quadric hypersurface, the projection-ramification map
\[ \rho \from \Gr(n, n+1) = \P^{n} \to {\P^n}^* \]
is again regular, and is in fact the duality isomorphism induced by the (non-degenerate) quadric $X$.
In particular, it has degree $1$.

The cases of the Veronese surface $X \cong \P^2 \subset \P^5$ and the quartic surface scroll $X = \P(\O(2) \oplus \O(2)) \subset \P^5$ are particularly delightful.
In these cases, the fibers of $\rho$ have an interpretation in terms of 2-torsion points of certain elliptic curves, which we now describe.
For the Veronese surface, the target of $\rho$ is the linear series of cubics in $\P^2$.
The points of fiber of $\rho$ over a cubic $R \subset \P^2$ correspond naturally to the non-trivial 2-torsion points of $\Pic R$.
In particular, the degree of $\rho$ is $3$.
For the quartic surface scroll, the target of $\rho$ modulo the action of $\Aut X$ is birational to the moduli space of $(R, \eta)$ where $R$ is a plane cubic and $\eta$ is a non-trivial 2-torsion point of $\Pic R$.
The points of the fiber of $\rho$ over $(R, \eta)$ correspond naturally to the two elements of $\Pic E[2] \setminus \pi^* \Pic R [2]$, where $E \to R$ is the \'etale double cover defined by $\eta$.
In particular, the degree of $\rho$ is $2$.
In this case, the source of $\rho$ (the Grassmannian $\Gr(3, H^0(X, \O(1)))$ modulo the action of $\Aut X$) has several known moduli interpretations.
It is birational to the moduli of unordered triplets of unordered pairs of points on $\P^1$, namely $M_{0,6}/(S_2 \times S_2 \times S_2 \rtimes S_3)$.
This space, in turn, is isomorphic to the moduli of hyperelliptic curves with a maximal isotropic subspace of the $\F_2$-vector space of $2$-torsion points, or equivalently, to the moduli of principally polarized abelian surfaces with a maximal isotropic subspace of the $\F_2$-vector space of $2$-torsion points \cite[Example~4.2]{dol.how:15}.
The involution on this space induced by the $2$-to-$1$ map $\rho$ coincides with the classical Richelot or Fricke involution \cite[Remark~4.3]{dol.how:15}.
See \autoref{rem:richelot} for more details.

The following result summarizes our knowledge of the degree of $\rho$.
\begin{maintheorem}\label{thm:examples}
  Let $\rho$ be the projection-ramification map for $X \subset \P^n$ of minimal degree.
\begin{enumerate}
  \item If $X \subset \P^{n}$ is a rational normal curve, then $\rho$ is regular and $\deg \rho = \frac{(2n-2)!}{n!(n-1)!}$.
  \item  If $X \subset \P^{n}$ is a quadric hypersurface, then $\rho$ is an isomorphism; in particular, $\deg \rho = 1$.

  \item  If  $X = \P^{r-1} \times \P^{1} \hookrightarrow \P^{2r-1}$ is the Segre embedding, then $\deg \rho = 1$.

  \item  If $X \subset \P^{5}$ is the Veronese surface, then $ \deg \rho = 3$.
  \item If $X \subset \P^{5}$ is a general quartic surface scroll, then $\deg \rho = 2$.
  \item If $X = \P(\O_{\P^{1}}(1) \oplus \O_{\P^{1}}(k+1)) \subset \P^{k+3}$ is the surface scroll with the most imbalanced splitting type, then $\deg \rho = 1$.
  \item If $X = \P(\O_{\P^{1}}(1) \oplus \O_{\P^{1}}(1) \oplus \O_{\P^{1}}(k+1)) \subset \P^{k+5}$ is the threefold scroll with the most imbalanced splitting type, then $\deg \rho = 1$.
\end{enumerate} 
\end{maintheorem}
In the main text, the items in \autoref{thm:examples} are treated in \autoref{sec:arnc}, \autoref{sec:aquadricsurface}, \autoref{prop:segre}, \autoref{prop:veronese}, \autoref{prop:quarticscroll}, \autoref{prop:eccentric_surface}, and \autoref{prop:eccentric_threefold}, respectively.

\subsection{Further remarks}
There are two natural enumerative problems regarding finite coverings of curves.
The first problem, originating in the work of Hurwitz, is to compute the number of branched covers $C \to \P^1$ with a specified set of branch points $B \subset \P^1$.
These  Hurwitz numbers are difficult to compute, but they exhibit remarkable structure \cite{eke.lan.sha.ea:99,eke.lan.sha.ea:01}.
The second problem is to compute the number of maps $C \to \P^1$ with a prescribed set of ramification points $R \subset C$.  In fact, it is easy to see that this problem is only meaningful when $C = \P^{1}$, in which case it is immediately and easily answered by Schubert calculus, yielding the Catalan numbers.

In higher dimensions, however, the analogue of the Hurwitz problem is expected to be much less interesting, as evidenced by Chisini's conjecture (proved by Kulikov \cite{kul:08}).
A branched cover $S \to \P^2$ with generic branching is uniquely determined by its branch divisor $B \subset \P^2$, with finitely many well-understood counterexamples.
In contrast, as hinted by \autoref{thm:minimaldegree}, the enumerative problem regarding the ramification divisor persists, and poses a significant challenge.
In some sense, the enumerative problems regarding the branch and ramification divisors trade places, certainly in terms of difficulty, but hopefully also in terms of structure.

\subsection{Further questions}
Our work raises several questions, some of which we hope to return to in the future.
\subsubsection{The enumerative problem for scrolls}\label{sec:qscroll}
Recall that every vector bundle on $\P^1$ is isomorphic to a direct sum of line bundles.
In particular, an ample vector bundle of rank $r$ and degree $d$ is isomorphic to $\O(a_1) \oplus \cdots \oplus \O(a_r)$ for positive integers $a_1, \dots, a_r$ satisfying $a_1 \leq \cdots \leq a_r$ and $a_1 + \cdots + a_r = d$.
It is thus specified up to isomorphism by an $r$-term partition of $d$.
Let $\Sigma_{r,d}$ be the set of $r$-term partitions of $d$.
We get a function $\phi \from \Sigma_{r,d} \to \Z_{\geq 0}$ defined by
\[ \phi(a_1, \dots, a_r) = \text{Degree of the projection-ramification map for $X \subset \P^{r+d}$},\]
where $X = \P\left( \O(a_1) \oplus \dots \oplus \O(a_r) \right)$ is embedded in $\P^{r+d}$ by $\O_X(1)$.
The set $\Sigma_{r,d}$ is partially ordered by the dominance order $\prec$.
In terms of vector bundles, $\prec$ translates into isotrivial specialization: $(a_1, \dots, a_r) \prec (b_1, \dots, b_r)$ if and only if $\O(b_1) \oplus \dots \oplus \O(b_r)$ isotrivially specializes to $\O(a_1) \oplus \dots \oplus \O(a_r)$.
In this case, by the lower semi-continuity of degrees of rational maps, we get
\[ \phi(a_1, \dots, a_r) \leq \phi(b_1,\dots, b_r).\]
Thus, $\phi$ is order preserving.

We hope that the enumerative function $\phi \from \Sigma_{r,d} \to \Z$ admits a deeper structure, such as a recurrence relation or generating function.
\autoref{thm:minimaldegree} and \autoref{thm:examples} only scratch the surface as far as $\phi$ is concerned.
\autoref{thm:minimaldegree} states that $\phi$ is not identically zero, at least if $d$ is sufficiently large.
\autoref{thm:examples} computes $\phi$ for the partitions $(d)$, $(1, \dots, 1)$, $(1, k+1)$, $(1,1,k+1)$, and $(2,2)$.
Some more examples, calculated using randomized trials over finite fields in \texttt{Macaulay2} \cite{gra.sti:} and \texttt{Magma} \cite{bos.can.pla:97}, are tabulated in \autoref{tab:computation}.
There seems to be some enchanting combinatorics behind $\phi$. As a sample, we point out that the sequence of numbers $\phi(n)$ are the Catalan numbers, and the sequence $\phi(\lfloor n/2 \rfloor, \lceil n/2 \rceil)$ (appearing down the diagonal in \autoref{tab:computation}: $1, 2, 6, 22, 92, 422, \dots$) seems to be \cite[A001181]{oeis:}, namely the number of Baxter permutations on $n-2$ letters.
We plan on conducting a more complete enumerative investigation of $\phi$ in a future paper. 
\begin{table}
  \centering
  \rowcolors{2}{gray!10}{white}
  \caption{Degrees of the projection-ramification maps for $X = \P(\O(a_1) \oplus \O(a_2))$} \label{tab:computation}
  \begin{tabular}{l| r r r r}
    \rowcolor{gray!25}
    \diagbox{$a_1$}{$a_2$} & 1 & 2 & 3 & 4\\
    \hline
    1 & 1 & & &\\
    2 & 1 & 2 & &\\
    3 & 1 & 6 & 22 &\\
    4 & 1 & 17 & 92 & 422\\
  \end{tabular}
\end{table}

\subsubsection{Non-maximal variation} Given that our counterexamples to maximal variation are all scrolls over curves, it is natural to wonder whether failure of maximal variation can only occur for scrolls.  Ideally, we would seek a clean classification of varieties failing maximal variation -- a good start would be to precisely classify all  rational normal scrolls which fail maximal variation.

It would especially be nice to establish maximal variation for all threefolds -- by \autoref{thm:main} along with the well-known fact that the only threefolds which have degenerate duals are scrolls, we need only establish maximal variation for scrolls, i.e. $\P^{2}$-bundles over curves.

\subsubsection{Compressible varieties}  Another problem which naturally emerges from our work is to classify those varieties which are compressible.  For example, it is easy to see that smooth, non-degenerate complete intersection varieties are automatically incompressible.  Therefore, it might be possible to prove incompressibility for varieties of small codimension in large projective spaces.  It may also be possible to classify the compressible varieties having a fixed small codimension, starting with codimension two.

\subsubsection{Positive characteristics}
The analysis of maximal variation and the computation of the degree of $\rho$  will surely bring new surprises and require new techniques in positive characteristics.
We do not know if \autoref{thm:main} or \autoref{thm:minimaldegree} holds in positive characteristic; our proofs certainly do not work.
The degrees in \autoref{thm:examples}, and likewise the values of the enumerative function $\phi \from \Sigma_{r,d} \to \Z$ defined in \autoref{sec:qscroll}, depend on the characteristic due to the presence of inseparable covers.
Indeed, this is true even for rational normal curves \cite{oss:06}.

\subsubsection{Picture over the real numbers}
Consider the projection-ramification map of a rational normal curve of degree $n$, also called the Wronskian map,
\[ \rho \from \Gr(2, n+1) \to \P^{2n-2}.\]
The real algebraic geometry surrounding $\rho$ plays an important role in real enumerative geometry, the theory of real algebraic curves, and control theory, thanks to the B. and M. Shapiro conjecture.
Proved by Eremenko and Gabrielov, this conjecture states that if $L \in \Gr(2, n+1)$ is such that the ramification divisor $\rho(L)$ is the sum of $(2n-2)$ real points in $\P^1$, then $L$ is a real point of $\Gr(2, n+1)$ \cite{sot:00, ere.gab:02}.
\autoref{thm:minimaldegree} potentially sets the stage for a higher-dimensional generalization of the body of work around the Shapiro conjecture.
In particular, it would be interesting to find a uniform topological picture explaining the numbers $\deg \rho$, similar to the ``nets'' introduced by Eremenko and Gabrielov which elegantly explain the appearance of Catalan numbers.

\subsection{Notation and conventions}
We work over an algebraically closed field $\k$ of characteristic zero, not simply for convenience, but for necessity---we appeal to Bertini's theorem, generic smoothness, and Kodaira vanishing.
All schemes are of finite type over $\k$.
A variety is a separated integral scheme.
For a scheme $X$, we let $X^{\rm sm} \subset X$ be the smooth locus.

We go back and forth without comment between divisors and line bundles, and likewise, between locally free sheaves and vector bundles.
We follow Grothendieck's convention for projectivization.
That is, the projectivization $\P E$ of a vector bundle $E$ is the space of one dimensional quotients of $E$.
For a line bundle $L$ on $X$, we denote by $|L|$ the projective space $\P H^0(X, L)^*$.

Given a vector bundle $F$ on $X$, we denote by $P(F)$ the sheaf of principal parts of $F$.
This is defined by the formula
\[ P(F) = {\pi_2}_* \left( {\pi_1}^* F \otimes \mathcal \O_{X \times X}/ I_{\Delta}^2 \right),\]
where the $\pi_i$ are the projections on the two factors and $\Delta \subset X \times X$ is the diagonal.
We remind the reader of the exact sequence
\[ 0 \to F \otimes \Omega_S \to P(F) \to F \to 0\]
and the evaluation map
\[ e \from H^0(X, F) \otimes \O_X \to P(F).\]

\subsection{Organization}
In \autoref{sec:prmap}, we give basic definitions, culminating in the precise general definition of $\rho$ (\autoref{def:ProjectionRamification}).
The subsequence sections are logically independent of each other and can be read in any order after \autoref{sec:prmap}.

In \autoref{sec:proof_of_theorem:main}, we prove \autoref{thm:main}\eqref{item:incomp} (\autoref{prop:incompress}).
We then introduce the notion of non-defectivity, which generalizes the condition of having a divisorial dual.
After establishing basic properties of non-defectivity, we prove \autoref{thm:main}\eqref{item:dual} (\autoref{thm:mainMain}).

In \autoref{sec:minimaldegree}, we prove \autoref{thm:minimaldegree} (\autoref{thm:actualminimaldegree}).
In the same section, we derive explicit formulas for the ramification divisors for scrolls in \autoref{sec:prscrolls}, give the examples advertised in \autoref{thm:counterexamples} (\autoref{sec:failure}), and treat the threefold scrolls in \autoref{sec:eccentric_threefolds}.

In \autoref{sec:generic}, the main goal is the proof of \autoref{thm:rationalnormalscrolls} (\autoref{thm:actualrationalnormalscrolls}).
For this, we lay the groundwork by doing some low degree cases by hand (\autoref{sec:lowdegree}).
We then recall the theory of (linked) limit linear series for vector bundles of higher rank in \autoref{sec:lls}, and define the projection-ramification map for linked linear series in \autoref{sec:prnongeneric} and \autoref{sec:prlls}.
By a degeneration argument involving linked linear series, we prove \autoref{thm:rationalnormalscrolls} in \autoref{sec:llsproof}.

In \autoref{sec:enumerativeproblems}, we turn to the enumerative problem of finding the degree of $\rho$, namely the results in \autoref{thm:examples}.
We treat the cases of rational normal curves and quadric hypersurfaces quickly in \autoref{sec:arnc} and \autoref{sec:aquadricsurface}.
We devote \autoref{sec:veronese} to the case of the Veronese surface and \autoref{sec:quartic_scroll} to the case of the quartic surface scroll.

\subsection*{Acknowledgments}
First and foremost, we thank Fyodor Zak for generously sharing numerous useful comments, ideas, and encouragement over the span of several months.
We also thank Izzet Coskun, Joe Harris, Mirella Manaresi, Brian Osserman, and Dennis Tseng for useful conversations. 
A.D. thanks the Australian Research Council for the grant number \texttt{DE180101360} that supported a part of this project.
A.D. and A.P. conducted a part of this research at the Banff International Research Center while attending the workshop titled \emph{Moduli spaces, birational geometry, and wall crossings} organized by Dan Abramovich, Jim Bryan, and Dawei Chen, and are grateful for the opportunity to attend.

\section{The projection-ramification map}\label{sec:prmap}
In this section, we define a projection-ramification map for a linear series on a proper, normal, variety $X$.
For $X \subset \P^n$, taking the linear series cut out by the hyperplanes recovers the projection-ramification map introduced in \autoref{sec:intro}.
Working with abstract linear series, however, offers more flexibility that is helpful in inductive proofs.

Let $X$ be a proper variety of dimension $r$ over an algebraically closed field $k$ of characteristic zero.
A \emph{linear series} on $X$ is a pair $(L, W)$ consisting of a line bundle $L$ on $X$ and a subspace $W \subset H^0(X, L)$.
The \emph{complete linear series} associated to $L$ is $(L, W)$ with $W = H^0(X, L)$.
A \emph{projection} is a linear series $(L, V)$ with $\dim V = r+1$.
A \emph{projection of $(L, W)$} is a projection $(L, V)$ with $V \subset W$.
As a convention, we use $V$ for projections and $W$ for more general linear series.

\begin{definition}[Properly ramified projection]
  \label{def:properlyramified}
We say that a projection $(L,V)$ is \emph{properly ramified} if the evaluation homomorphism
\[e \from V \otimes \O_{X} \to P(L)\]
is an isomorphism over a general point in $X$.  If $(L,V)$ is properly ramified, its \emph{ramification divisor}
\[R(L,V) \subset X\]
is the closure of the scheme defined by the determinant of $e \from V \otimes \O_{X^{\rm sm}} \to P(L)|_{X^{\rm sm}}$.
\end{definition}
In most cases, $L$ is clear from context, so we drop it from the notation and denote the ramification divisor simply by $R(V)$.
\begin{remark}\label{rem:Jacobian}
  Suppose for simplicity that $V$ is a base-point free linear series that yields a surjective map $\phi \from X \to \P V$.
  Then the ramification divisor may be defined as the degeneracy locus of the map
  \[ d \phi \from T_{X} \to \phi^* T_{\P V}\]
  on tangent spaces.
  The degeneracy locus is the zero locus of $\det \phi$, which in local coordinates, is given by the determinant of the Jacobian matrix $\left( \frac{\partial \phi_i}{\partial x_j} \right)$.
  Therefore, the ramification divisor $R(L, V)$ is also often called the \emph{Jacobian} of the linear series $(L, V)$ (see, for example, \cite[1.1.7]{dol:12}).
\end{remark}

A projection $(L, V)$ gives the evaluation map
\[e \from V \otimes \O_X \to L.\]
The evaluation map yields a map $p_{V,L} \from X \dashrightarrow \P V$, regular on the non-empty open set of $X$ where $e$ is surjective.
The following is an easy observation, whose proof we skip.
\begin{proposition}\label{prop:proj}
  The projection $(L, V)$ is properly ramified if and only if the map on tangent spaces induced by $p_{V,L}$ is generically an isomorphism.
  In characteristic zero, this is equivalent to the condition that $p_{V,L}$ is dominant.
\end{proposition}

For a fixed $(L, W)$, the set of all projections of $(L, W)$ are parametrized by the Grassmannian $\Gr(r+1, W)$.
The property of being properly ramified is a Zariski open condition on the Grassmannian.

We now define a map that assigns to a projection its ramification divisor.
To do so, we interpret the ramification divisor as an element of a linear series.

Assume, furthermore, that $X$ is normal.
Let $K_X$ be the canonical sheaf of $X$.
Denoting by $i \from X^{\rm sm} \to X$ the inclusion, $K_X$ is given by the push-forward
\[ K_X = i_* K_{X^{\rm sm}}.\]
Note that, since $X$ is normal, the complement of $X^{\rm sm} \subset X$ has codimension at least 2.
The sheaf $K_X$ is coherent, reflexive, and satisfies Serre's S2 condition.

Let $L$ be a line bundle on $X$.
The sheaf $P(L)$ is locally free of rank $(r+1)$ on $X^{\rm sm}$, and we have a canonical isomorphism
\[ \bigwedge^{r+1} P(L) |_{X^{\rm sm}} \cong K_{X^{\rm sm}} \otimes L^{r+1}.\]
Given a subspace $V \subset H^0(X, L)$, we apply $\bigwedge^{r+1}$ to the evaluation map
\[ e \from V \otimes \O_{X^{\rm sm}} \to P(L)|_{X^{\rm sm}},\]
to get
\[ \det e \from \det V \otimes \O_{X^{\rm sm}} \to K_{X^{\rm sm}} \otimes L^{r+1}. \]
By applying $i_*$ and taking global sections, we get
\begin{equation}\label{eqn:ramsection}
  r_V \from \det V \to H^0(X, K_X \otimes L^{r+1}).
\end{equation}
If $(L, V)$ is properly ramified, then this map is non-zero, and hence gives a point of the projective space $\P H^0(X, K_X \otimes L^{r+1})^*$.
Doing the same construction universally over the Grassmannian $\Gr = \Gr(r+1, W)$ yields a map
\begin{equation}\label{eqn:rammap}
  r \from \det \mathcal V \to H^0(X, K_X \otimes L^{r+1}) \otimes \O_{\Gr},
\end{equation}
where $\mathcal V \subset W \otimes \O_{\Gr}$ is the universal sub-bundle of rank $(r+1)$.
Let $U \subset \Gr$ be the open subset of properly ramified projections.
Then the map in \eqref{eqn:rammap} is non-zero at every point of $U$, and defines a map $U \to \P H^0(X, K_X \otimes L^{r+1})^*$ given by the surjection
\begin{equation}\label{eqn:rammapfamily}
  H^0(X, K_X \otimes L^{r+1})^* \otimes \O_{U} \to \det \mathcal V|_U^*.
\end{equation}
Note that $U$ is non-empty if and only if $W$ separates tangent vectors at a general point of $X$.
\begin{definition}[Projection-ramification map]
  \label{def:ProjectionRamification}
  Let $(L, W)$ be a linear series that separates tangent vectors at a general point of $X$.
  The \emph{projection-ramification} map for $(L,W)$ is the rational map
  \[
    \rho_{(X,L,W)} \from \Gr(r+1, W) \dashrightarrow \P H^0(X, K_X \otimes L^{r+1})^*
  \]
  defined on the non-empty open subset of properly ramified maps by \eqref{eqn:rammapfamily}.
\end{definition}
If any of $X$, $L$, or $W$ are clear from context, we drop them from the notation.
In particular, for a non-degenerate $X \subset \P^n$, we denote by $\rho_X$ the map $\rho_{X,L,W}$  with $L = \O_X(1)$ and $W$ the image in $H^0(X, L)$ of $H^0(\P^n, \O(1))$.

Note that the map \eqref{eqn:rammapfamily} factors as
\[ \det \mathcal V \xrightarrow{a} \bigwedge^{r+1} W \otimes \O_{\Gr} \xrightarrow{b} H^0(X, K_X \otimes L^{r+1}) \otimes \O_{\Gr},\]
where $a$ is $\wedge^{r+1}$ applied to the universal inclusion $\mathcal V \subset W \otimes \O_{\Gr}$, and $b$ is induced by $\wedge^{r+1}$ applied to the evaluation map $e \from W \otimes \O_{X} \to P(L)$.
The map $a$ defines the Pl\"ucker embedding
\[ i \from \Gr(r+1, W) \to \P \left(\bigwedge^{r+1}W^*\right),\]
and the map $b$ defines a linear projection
\[ p \from \P \left(\bigwedge^{r+1}W^*\right) \dashrightarrow \P H^0(X, K_X \otimes L^{r+1}).\]
Thus, $\rho_{X,L,W}$ factors as the Pl\"ucker embedding followed by a linear projection.

\section{Maximal variation for incompressible and non-defective $X$}
\label{sec:proof_of_theorem:main}
The goal of this section is to prove \autoref{thm:main}.
We begin by proving part \eqref{item:incomp}, which is substantially easier.
\begin{proposition}[\autoref{thm:main}~\eqref{item:incomp}]
  \label{prop:incompress}
  Let $X \subset \P^n$ be a non-degenerate, normal, incompressible projective variety over a field of characteristic zero.
  Then $\rho_X$ is a finite map.
\end{proposition}
\begin{proof}
  Set $L = \O(1)$ and let $W \subset H^0(X, L)$ be the image of $H^0(\P^n, \O(1))$.
  Let $V \subset W$ be an $(r+1)$-dimensional subspace.
  Since $X$ is incompressible, the projection map $p_{V,L} \from X \dashrightarrow \P V$ induced by $(L, V)$ is dominant.
  By \autoref{prop:proj}, this implies that $(L, V)$ is properly ramified.
  Since $V$ was arbitrary, the projection-ramification map 
  \[ \rho \from \Gr(r+1, W) \to |K_X + (r+1) H|\]
  is regular.
  Since the Picard rank of a Grassmannian is $1$, a regular map from a Grassmannian is either constant or finite.
  It is easy to check that $\rho$ is not constant; so it must be finite.
\end{proof}

For the proof of part \eqref{item:dual} of \autoref{thm:main}, we proceed inductively by showing that a general $(n-r-1)$-dimensional linear subspace which is incident to $X$ is an isolated point in its fiber under $\rho$.
Again, it is more convenient to work with the more abstract set-up of a linear series, allowing for series that are not very ample.

Let $X$ be a proper variety of dimension $r$, and let $(L, W)$ be a linear series on $X$.
For an ideal sheaf $I \subset \O_X$ we denote by $W \otimes I$ the subspace of $W$ consisting of the sections that vanish modulo $I$. More precisely, if $K$ is the kernel of the evaluation map
\[ W \otimes \O_X \to L \otimes \O_X/I,\]
then $W \otimes I = H^0(X, K)$.
In particular, for $W = H^0(X, L)$, we have $W \otimes I = H^0(X, L \otimes I)$.
For $s \in W \otimes I$, the vanishing locus $v(s)$ refers to the vanishing locus of $s$ as a section of $L$.
We set $|W| = \P W^*$, the space of one-dimensional subspaces of $W$, and likewise $|W \otimes I| = \P (W \otimes I)^*$.
For a complete linear series, we write $|L|$ for $|W|$.
Note that $v(s) = v(\lambda s)$ for a non-zero scalar $\lambda$, so it causes no ambiguity to talk about $v(s)$ for $s \in |W|$.

\subsection{Non-defective linear series}\label{sec:non-defectivity}
We study a positivity property of linear series that generalizes the property of having a divisorial dual.
\begin{definition}[Non-defective linear series]
  \label{def:genericallynon-defective} 
  We say that a linear series $(L, W)$ is \emph{non-defective} if,  for a general point $x \in X$ either $W \otimes \mathfrak m_x^2 = 0$, or there exists $s \in W \otimes \mathfrak m_x^2$ such that $v(s)$ has an isolated singularity at $x$.
\end{definition}
Note that for $s \in |W|$, the condition that $v(s)$ have an isolated singularity at $x$ is a Zariski open condition on $|W|$.
Therefore, if there exists an $s \in |W \otimes \mathfrak m_x^2|$ such that $v(s)$ has an isolated singularity at $x$, then a general $s \in |W \otimes \mathfrak m_x^2|$ has the same property.
\begin{remark}
  Let $x$ be a point of $X$.
  Suppose there exists $s \in |W|$ with an isolated singularity at $x$.
  It may be tempting to conclude from this that $(L, W)$ is non-defective.
  This is not necessarily true!
  For example, take $X = \F_3$.
  Denote by $E$ the section of self-intersection $-3$ and $F$ the fiber of the projection $\F_3 \to \P^1$.
  Let $L = \O_X(E + 2F)$ and $W = H^0(X, L)$.
  For $x \in E$, the general member of $|W \otimes \mathfrak m_x^2|$ has an isolated singularity at $x$, but the same is not true for a general $x \in X$.
\end{remark}

\begin{remark}
  Suppose $(L, W)$ is non-defective.
  Let $x \in X$ be general, and let $s \in |W|$ be such that $v(s)$ has an isolated singularity at $x$.
  For all such $s$, it may be the case $v(s)$ has singularities away from $x$, even along a positive dimensional locus.
  For example, let $\pi \from X \to \P^2$  be the blow-up at a point, and $E$ the exceptional divisor.
  The complete linear series associated to $L = \pi^* \O(2) \otimes \O(2E)$ is non-defective, but for every global section of $L$, the singular locus of $v(s)$ contains $E$.
\end{remark}

We now define the conormal variety of a linear series, which plays an important role in our analysis of non-defectivity.
Let $K$ be the kernel of the evaluation map
\[ e \from W \otimes \O_X \to P(L).\]
Let $U \subset X$ be an open subset such that $K|_U$ is locally free and the dual of the inclusion 
\[W^* \otimes \O_U \to K|_U^*\]
is a surjection.
This surjection defines a closed embedding $\P(K|_U) \subset U \times |W|$.
The \emph{conormal variety of $(L,W)$}, denoted by $P_{L,W}$, is the closure of $\P(K|_U)$ in $X \times |W|$.

\begin{proposition}\label{prop:dimension}
  \label{prop:dimP}
  Suppose $(L, W)$ is non-defective.
  If $\dim W \geq r+2$, then $P_{L,W}$ is irreducible of dimension $\dim W - 2$.
  If $\dim W \leq r+1$, then $P_{L,W}$ is empty.
\end{proposition} 

\begin{proof}
  Set $n = \dim |W| = \dim W - 1$.
  Let $k$ be the (generic) rank of $K$, namely the rank of the locally free sheaf $K|_U$.
  Then $k \geq n-r$.
  The statement of the proposition is equivalent to showing that if $k > 0$, then $k = n-r$.

  For brevity, set $P = P_{L,W}$.
  Consider the projection $\sigma \from P \to |W|$, obtained by restricting the second projection $X \times |W| \to |W|$.
  For $s \in |W|$, we view $\sigma^{-1}(s)$ as a subscheme of $X$.
  We then have
  \begin{align*}
    \sigma^{-1}(s) \cap U = \sing(v(s)) \cap U.
  \end{align*}

  Suppose $r>0$.
  Then $P$ is non-empty and irreducible, since it is the closure of a non-empty and irreducible variety.
  Since $(L,W)$ is non-defective, a general point $(x,s) \in P$ is such that $x$ is an isolated point of $\sing(v(s))$.
  Therefore, $\sigma \from P \to |W|$ is generically finite onto its image.
  We conclude that $\dim P \leq \dim |W|$, and hence $k \leq n-r+1$.

  To show that $k = n-r$, it suffices to show that $\sigma \from P \to |W|$ is not surjective.
  We do so using Bertini's theorem.
  Let $B \subset X$ denote the union of the base locus of $|W|$ and the singular locus of $X$.
  Then $B$ is a proper closed subset of $X$.
  Let $P^B \subset P$ be the pre-image of $B$ under the projection $\pi \from P \to X$.
  By the definition of $P$, the map $\pi \from P \to X$ is surjective, and hence $P^B$ is a proper closed subset of $P$.
  Since $P$ is irreducible, we have $\dim P^B < \dim P \leq \dim |W|$, so the projection $P^B \to |W|$ cannot be dominant.
  Let $s \in |W|$ be general, in particular, not in the image of $P^B \to |W|$.
  By Bertini's theorem $v(s)$ is non-singular away from $B$.
  Thus, for any $x \in X$, the point $(x, s) \in X \times |W|$ does not lie in $P$.
  For $x \in B$, this is because $s$ is not in the image of $P^B$, and for $x \not \in B$, this is because $v(s)$ is non-singular at $x$.
  We conclude that $s$ does not lie in the image of $P \to |W|$.
  Hence $P \to |W|$ is not surjective.
\end{proof} 

\begin{proposition}
  \label{prop:dimensionCriterion}
  Let $(L, W)$ be a linear series with $\dim W \geq r+2$, and let $P = P_L$ be its conormal variety.
  The projection $\sigma\from P \to |W|$ is generically finite onto its image if and only if $(L, W)$ is non-defective. 
\end{proposition}

\begin{proof}
  Since $\dim W \geq r+2$, the conormal variety $P = P_{L,W}$ is non-empty.
  Let $(x,s) \in P$ be a general point.
  We may assume that $x \in U$.
  Then $x$ is a singular point of $v(s)$, and it is an isolated singularity of $v(s)$ if and only if $(x,s)$ is an isolated point in the fiber of $\sigma \from P \to |W|$ over $s$.
  The conclusion follows.
\end{proof}

The following observation relates non-defectivity with the non-degeneracy of the dual.
\begin{proposition}\label{prop:non-deg-dual}
  Let $X \subset \P^n$ be a non-degenerate projective variety.  Let $L = \O_X(1)$ and $W \subset H^0(X, L)$ the image of $H^0(\P^n, \O(1))$.
  Then $(L, W)$ is non-defective if and only if the dual variety $X^* \subset {\P^n}^*$ is a hypersurface.
\end{proposition}
\begin{proof}
  Since $X \subset \P^n$ is not contained in a hyperplane, we have $\dim W = n+1 \geq r+1$.
  Since $(L, W)$ is very ample, it separates tangent vectors on $X$, so the evaluation map
  \[ e \from W \otimes \O_X \to P(L)  \]
  is surjective.
  It follows that the rank of the kernel is $n-r$, and hence
  \[ \dim P_{L,W} = (n-r - 1) + r = n-1.\]
  By definition, the dual variety $X^* \subset {\P^n}^* = |W|$ is the image of the conormal variety under the projection $P_{L,W} \to |W|$.
  By \autoref{prop:dimensionCriterion}, $(L, W)$ is non-defective if and only if $\dim X^* = n-1$.
\end{proof}

\begin{proposition}\label{prop:ordinarydoublepoint}
  Let $(L, W)$ be a non-defective linear series on $X$ with $\dim W \geq r+2$.
  Let $x \in X$ be a general point.
  Then there exists $s \in |W|$ such that $v(s)$ has an ordinary double point singularity at $x$.
\end{proposition}
\begin{proof}
  By \autoref{prop:dimensionCriterion}, the projection $\sigma\from P \to |W|$ is generically finite onto its image. 
  Let $(x,s) \in P$ be a general point.
  Since our ground field is of characteristic zero, we may assume that $P$ is smooth at $(x,s)$, that $x \in U \cap X^{\rm sm}$, and $\sigma \from P \to |W|$ is a local immersion at $(x,s)$.
  This implies that $x \in \sing(v(s))$ is isolated, and also that $x$ is a reduced point of the scheme $\sing(v(s))$.
  These two properties show that $v(s)$ possesses an ordinary double point at $x$.
  To see this, choose local coordinates $(x_{1}, ..., x_{n})$ so that the complete local ring ${\widehat{\O}_{X,x}}$ is isomorphic to $k\llbracket x_{1},\dots, x_{r}\rrbracket$.
  After choosing a local trivialization for $L$ around $x$, the section $s$ corresponds to a power series $s(x_1,\dots,x_r)$ contained in $\mathfrak m_x^2 \widehat \O_{X,x}$.
  The germ of $\sing(v(s))$ at $x$ is cut out by the power series $\frac{\partial s}{\partial x_1}, \dots, \frac{\partial s}{\partial x_r}$.
  Since the germ of $\sing(v(s))$ at $x$ is the reduced point $x$, we get that $\frac{\partial s}{\partial x_1}, \dots, \frac{\partial s}{\partial x_r}$ are linearly independent as elements of $\mathfrak m_x / \mathfrak m_x^2$.
  From this, it is easy to check that the tangent cone of $s(x_1, \dots, x_r)$ at $x$ is a non-degenerate quadric cone.
\end{proof}

\begin{proposition}
  \label{prop:genericSeparateTangents}
  If $(L, W)$ is a non-defective linear series with $\dim W \geq r+1$, then $W$ separates tangent vectors at a general point $x \in X$.
  That is, the evaluation map
  \[e_{x}\from W \otimes \O_X \to L/\mathfrak m_x^2 L\]
  is surjective for general $x \in X$.
\end{proposition}
\begin{proof}
  By the definition of $P(L)$, we have a natural isomorphism
  \[ P(L)|_x = L/\mathfrak m_x^2 L,\]
  so it suffices to show that the evaluation map
  \[ e \from W \otimes \O_X \to P(L)\]
  is surjective at $x$.
  Let $k$ be the generic rank of $K$, the kernel of $e$.
  From the proof of \autoref{prop:dimP}, we get
  \[  k = \dim W - r - 1.\]
  Since $(r+1)$ is the generic rank of $P(L)$, we conclude that $e$ is generically surjective.
\end{proof}
\begin{corollary}\label{cor:properlyramified}
  Suppose $(L, W)$ is a non-defective linear series on $X$ with $\dim W \geq r+1$.
  Then there exists a properly ramified projection $(L,V)$ of $(L, W)$.
\end{corollary}
\begin{proof}
  This follows immediately from \autoref{prop:genericSeparateTangents}.
\end{proof}
As a consequence of \autoref{cor:properlyramified}, the projection-ramification rational map $\rho_{X,L, W}$ is defined for a non-defective linear series $(L, W)$ with $\dim W \geq r+1$.

Let $\pi \from \widetilde X \to X$ be the blow-up at a point $x \in X$, and $E \subset \widetilde X$ the exceptional divisor.
A linear series $(L, W)$ on $X$ gives a linear series $(\widetilde L, \widetilde W)$ as follows.
Take $\widetilde L = \pi^* L \otimes \O_{\widetilde X}(-E)$.
Note that $H^0(X, L) = H^0(\widetilde X, \pi^*L)$, so we may think of $W$ as a subspace of $H^0(\widetilde X, \pi^*L)$.
Take $\widetilde W = W \otimes \O_{\widetilde X}(-E)$ with its natural inclusion $\widetilde W \subset H^0(\widetilde X, \widetilde L)$.
\begin{proposition}
  \label{prop:blowuppoint}
  In the setup above, if $(L, W)$ is non-defective, $\dim W \geq r+2$, and $x \in X$ is general, then $(\widetilde L, \widetilde W)$ is also non-defective.
\end{proposition}

\begin{proof}
  Let $y$ be a general point of $\widetilde X$.
  We have the equality
  \[ \widetilde W \otimes \mathfrak m_y^2 = W \otimes \mathfrak m_x \cdot \mathfrak m_y^2. \]
  By \autoref{prop:genericSeparateTangents}, for a general $y \in X$, we have
  \[ \dim (W \otimes \mathfrak m_y^2) = \dim W - (r+1).\]
  Since $x \in X$ is general, we get
  \[ \dim (W \otimes \mathfrak m_x \cdot \mathfrak m_y^2) = \dim W - (r+2).\]
  If $\dim W = r+2$, then we get $\widetilde W \otimes \mathfrak m_y^2 = 0$, so we are done.
  Assume that $\dim W \geq r+3$.
  Then $\dim (W \otimes \mathfrak m_y^2) \geq 2$.
  Since $(L, W)$ is non-defective, a general $s \in W \otimes \mathfrak m_y^2$ is such that $v(s)$ has an isolated singularity at $y$.
  Moreover, since $\dim (W \otimes \mathfrak m_y^2) \geq 2$, for every $x \in X$, there exists $s \in V$ such that $v(s)$ passes through $x$.
  Hence, as $x \in X$ is general, there exists $s \in W \otimes \mathfrak m_y^2$ such that $v(s)$ has an isolated singularity at $y$ and passes through $x$.
  That is, there exists $s \in \widetilde W \otimes \mathfrak m_y^2 $ that has an isolated singularity at $y$.
  We conclude that $(\widetilde L, \widetilde W)$ is non-defective.
\end{proof}

\subsection{Maximal variation for non-defective pairs}\label{sec:nondefective}

In this section, we prove part~\eqref{item:dual} of \autoref{thm:main}.
In fact, we prove a more general result (\autoref{thm:mainMain}).

As before, $X$ is a proper, normal variety of dimension $r$ over an algebraically closed field of characteristic zero.
\begin{theorem}
  \label{thm:mainMain}
  Let $(L, W)$ be a non-defective linear series on $X$ with $\dim W \geq r+2$.
  Then the projection-ramification map $\rho_{X,L,W}$ is generically finite onto its image.
\end{theorem}

For the proof, we need two lemmas, which are essentially local computations.
Throughout, $X$, $L$, and $W$ are as in the statement of \autoref{thm:mainMain}.

\begin{lemma}\label{lem:tangentconeram}
  Let $x \in X$ be a general point and $V \subset W \otimes \mathfrak m_x$ a general $(r+1)$-dimensional subspace.
  Then $V$ is properly ramified, and the ramification divisor $R(V)$ has an ordinary double point singularity at $x$.
\end{lemma}

\begin{proof}
  Using \autoref{prop:ordinarydoublepoint} and \autoref{prop:genericSeparateTangents}, we get a basis $(s_{1}, ..., s_{n}, t)$ of $V$ satisfying the following two conditions:
  \begin{enumerate}
      \item $s_{1}, \dots, s_{n}$ generate $L \otimes ({\mathfrak m}_{x}/{\mathfrak m}^{2}_{x})$, and
      \item $v(t)$ has an ordinary double point singularity at $x$.
    \end{enumerate}  

    Let $\widehat{\O}_{X,x}$ denote the completion of the local ring at $x \in X$ along its maximal ideal.  Upon trivializing $L$, we may regard $s_{i}$ and $t$ as elements of $\widehat{\O}_{X,x}$, and can also assume  $\widehat{\O}_{X,x} = k\llbracket s_{1}, \dots s_{n}\rrbracket$.
    In the bases $(s_1, \dots, s_n, t)$ for $V$ and $(1, s_1, \dots, s_n)$ for $P(L)$, the evaluation map 
\begin{align*}
  e\from V \otimes \widehat{\O}_{X,x} \to P(L) \otimes \widehat{\O}_{X,x}
\end{align*}
has the matrix
\begin{align}\label{matrix}
\begin{pmatrix}
  s_{1} & s_{2} & \dots & t \\
  1 & 0 & \dots & \partial_{1}t \\
  0 & 1 & \dots & \partial_{2}t \\
  \vdots & \vdots & \vdots & \vdots \\
  0 & 0 & \dots & \partial_{n}t
\end{pmatrix},
\end{align}
where $\partial_{i}$ denotes $\frac{\partial}{\partial s_{i}}$.
The determinant of the matrix \eqref{matrix}
\begin{align*}
  t - \sum_{i}s_{i}\partial_{i}t
\end{align*}
is an analytic local equation for the ramification divisor $R(V)$ near $x$.
By applying the Euler identity to the degree 2 part of $t$ (expressed as a power series in the $s_i$), we see that $R(V)$ shares the same tangent cone as $v(t)$ at $x$.
The proposition follows.
\end{proof}

\begin{lemma}\label{lem:basepointfree}
  Let $x \in X$ be a general point and $V \subset W$ an $(r+1)$-dimensional subspace with a basis $(u, a_1,\dots, a_{r-1},b)$ where
  \begin{enumerate}
    \item $u$ does not vanish at $x$,
    \item $a_1, \dots , a_{r-1}$ vanish at $x$, and reduce to linearly independent elements of $L \otimes ({\mathfrak m}_{x}/{\mathfrak m}^{2}_{x})$, and
    \item $v(b)$ has an ordinary double point at $x$. 
  \end{enumerate}
  Then $R(V)$ contains $x$ and is smooth at $x$.
\end{lemma}

\begin{proof}
  That $R(V)$ contains $x$ is clear since $V \otimes \mathfrak{m}^{2}_{x} \neq 0$.

  For smoothness, we again work in the completion $\widehat{\O}_{X,x}$.
  After trivializing $L$, we assume $u, a_{1}, ..., b$ are elements of $\widehat{\O}_{X,x}$.
  We choose an element $z \in \widehat{\O}_{X,x}$ such that $(a_1, \dots, a_{r-1}, z)$ forms a system of coordinates, that is $\widehat{\O}_{X,x} \cong k\llbracket a_{1}, \dots , a_{r-1}, z \rrbracket$.
  With respect to the given basis of $V$ and the basis $1, a_1, \dots, a_{r-1}, z$ for $P(L)$, the evaluation map
  \begin{align*}
  e\from V \otimes \widehat{\O}_{X,x} \to P(L) \otimes \widehat{\O}_{X,x}
  \end{align*}
  has the matrix
\begin{align}\label{matrix2}
\begin{pmatrix}
  u & a_{1} & a_{2} & \dots & b \\
  \partial_{1}u & 1 & 0 & \dots & \partial_{1}b \\
  \partial_{2}u & 0 & 1 & \dots & \partial_{2}b \\
  \vdots & \vdots & \vdots & \vdots \\
  \partial_{z}u  & 0 & 0 & \dots & \partial_{z}b
\end{pmatrix}
\end{align}
The determinant of the matrix \eqref{matrix2} is the analytic local equation for $R(V)$.
It is given by
\begin{align*}
   \bar{u} \cdot \partial_{z}b \pm \partial_{z}u \cdot \bar{b},
 \end{align*} 
 where, for $r \in \widehat{\O}_{X,x}$ we set
 \[\bar{r} = r - a_{1}\partial_{1}r - a_{2}\partial_{2}r - \dots - z \partial_{z} r.\]
 Since $b \in {\mathfrak m}^{2}_{x}$, we get that $\bar{b} \in {\mathfrak m}^{2}_{x}$, and so $\partial_{z}b \in {\mathfrak m}_{x}$.
 Furthermore, since the tangent cone of $b$ is a non-degenerate quadric, we also get that $\partial_z b \not \in \mathfrak m_x^2$.
 Since $\overline{u}$ is a unit, we see that the tangent cone of $R(V)$ at $x$ is the hyperplane cut out by $\partial_z b \in \mathfrak m_x/\mathfrak m_x^2$.
 So $R(V)$ is smooth at $x$.
\end{proof}

We now have all the tools for the proof of \autoref{thm:mainMain}. 
\begin{proof}[Proof of \autoref{thm:mainMain}]
  We induct on $\dim W$.
  The base case $\dim W = r+1$ is clear.

  We now do the induction step.
  Suppose $\dim W \geq r+2$.
  Choose a general point $x \in X$ such that the induced linear series $(\widetilde L, \widetilde W)$ on $\widetilde X = \Bl_x X$ is non-defective as in \autoref{prop:blowuppoint}.
  Choose a general $(r+1)$-dimensional subspace $V \subset W \otimes \mathfrak m_x = \widetilde W$ that satisfies the hypotheses of \autoref{lem:tangentconeram}.
  By the induction hypothesis, $V$ considered as a projection of $(\widetilde L, \widetilde W)$ is an isolated point in the projection-ramification map for $\widetilde X$.
  We now show that it is also an isolated point in the projection-ramification map for $X$.

  Let $(C, 0)$ be a pointed smooth curve and $V \subset W \otimes \O_C$ a sub-bundle of rank $(r+1)$ such that
  \begin{enumerate}
  \item $V_{0} = V$, and 
  \item $V_{c} \neq V_{0}$ for $c \in C \setminus \{0\}$.
  \end{enumerate}
  We must show that $R(V_c) \neq R(V)$ for a general $c \in C$.

  Suppose $V_c \subset W \otimes \mathfrak m_x = \widetilde W$ for all $c \in C$.
  Denote by $\widetilde R(V_c)$ the ramification divisor of $V_c$ considered as a projection of $\widetilde X$.
  Since $V = V_0$ is an isolated point in the projection-ramification map for $\widetilde X$, we know that $\widetilde R(V_c) \neq \widetilde R(V_0)$ for a general $c \in C$.
  Clearly, $R(V_c)$ and $\widetilde R(V_c)$ agree away from the exceptional divisor, and hence we conclude that $R(V_c) \neq R(V_0)$ for a general $c \in C$.

  On the other hand, suppose $V_c \not \subset W \otimes \mathfrak m_x = \widetilde W$ for a general $c \in C$.
  Consider the evaluation maps
  \[ e_c \from V_c \to L / \mathfrak m_x^2 L \]
  between an $(r+1)$-dimensional source and $(r+1)$-dimensional target.
  Since $V = V_0$ satisfies the hypotheses of \autoref{lem:tangentconeram}, $\rk e_0 = r$.
  Therefore, by semi-continuity, $\rk e_c \geq r$ for all $c \in C$.
  If $\rk e_c = (r+1)$ for a general $c \in C$, then $x \not \in R(V_c)$, and hence $R(V_c) \neq R(V)$.
  Otherwise, by shrinking $C$ if necessary, assume $\rk e_c = r$ for all $c \in C$.
  In other words, $\dim (V_c \otimes \mathfrak m_x^2) = 1$ for all $c \in C$.
  Let $b_c \in V_c \otimes \mathfrak m_x^2$ be a non-zero element.
  Since $v(b_0)$ has an ordinary double-point singularity at $x$, so does $v(b_c)$.
  Also, since $\rk (e_c) = r$ and $V_c \not \in W \otimes \mathfrak m_x$ for a general $c$, there exists $u_c \in V_c$ not vanishing at $x$, and a set of $(r-1)$ other elements that vanish at $x$ but reduce to linearly independent elements modulo $\mathfrak m_x^2$.
  That is, $V_c$ satisfies the hypotheses of \autoref{lem:basepointfree} for a general $c \in C$.
  But \autoref{lem:basepointfree} implies that $R(V_c)$ is smooth at $x$.
  Since $R(V_0)$ is singular at $x$, we conclude that $R(V_0) \neq R(V_c)$.
  The induction step is now complete.
\end{proof}

We immediately get part~\eqref{item:dual} of \autoref{thm:main}.
\begin{corollary}
  \label{cor:maintheorem} Let $X \subset \P^{n}$ be a non-degenerate projective variety such that the dual variety $X^{*} \subset \P^{n*}$ is a hypersurface. Then $\rho_{X}$ is generically finite onto its image.
\end{corollary}
\begin{proof}
  By \autoref{prop:non-deg-dual} the linear series on $X$ that gives the embedding $X \subset \P^n$ is non-defective.
  Now apply \autoref{thm:mainMain}.
\end{proof}

\begin{corollary}\label{cor:lowdim}
  Let $X \subset \P^n$ be a non-degenerate smooth curve or a surface.
  Then $\rho_X$ is generically finite onto its image.
\end{corollary}
\begin{proof}
  Curves and surfaces have divisorial duals, so \autoref{cor:maintheorem} applies.
\end{proof}

\section{Projection-ramification for varieties of minimal degree}\label{sec:minimaldegree}
In this section, we prove \autoref{thm:minimaldegree}, which relates varieties of minimal degree and the projection-ramification map.
We then prove \autoref{thm:counterexamples} by constructing examples of rational scrolls where maximal variation fails.
Finally, we obtain an alternate and more explicit description of the projection-ramification map for scrolls, which we use repeatedly.

The following is an easy application of the Kodaira vanishing theorem.
\begin{proposition}\label{lem:kymh}
  Let $X \subset \P^n$ be a non-degenerate, smooth, projective, variety of dimension $r \geq 1$ over a field of characteristic zero.
  For all $m \geq r$, we have the inequality
  \begin{equation}\label{eqn:KYmH}
    {m \choose r}(n-r) + {{m-1} \choose {r}}\leq h^0(X, K_X + mH).
  \end{equation}
  If equality holds for any $m \geq r$, then $X$ is a variety of minimal degree, that is $\deg X = n-r+1$.
  Conversely, for a variety of minimal degree, equality holds for all $m \geq r$.
\end{proposition}

\begin{proof}
  Without loss of generality, $X$ is embedded by the complete linear series.
  Indeed, passing to the complete linear series only increases the left side of the desired inequality, and does not change the right side.
  
  We first prove the inequality \eqref{eqn:KYmH}, using a double induction--first on $r$, and then on $m$.
  For the base case $r = 1$, Riemann--Roch gives
  \begin{equation}\label{eqn:r1}
    h^0(X, K_X + mH) = g_X - 1 + mn,
  \end{equation}
  from which \eqref{eqn:KYmH} follows for all $m$.

  Assume that \eqref{eqn:KYmH} holds for varieties of dimension $(r-1)$ and all $m \geq r-1$.
  Let $D \subset X$ be a general member of the linear series $|H|$.
  By Bertini's theorem, $D$ is a smooth variety.
  The adjunction formula $K_D = (K_X + H)|_D$ yields the exact sequence
  \begin{equation}\label{eqn:mainexact}
    0 \to \O_X(K_X + (m-1)H) \to \O_X(K_X + mH) \to \O_D(K_D + (m-1)H) \to 0.
  \end{equation}
  Note that, by the Kodaira vanishing theorem, we have $h^1(K_X + nH) = 0$ for all $n > 1$; we use this repeatedly, without further comment.
  For $m = r$, the long exact sequence in cohomology associated to \eqref{eqn:mainexact} gives
  \[ h^0(K_D + (r-1)H) \leq h^0(K_X + rH).\]
  By applying the induction hypothesis to $D$, we have
  \begin{equation}
    n-r \leq h^0(K_D + (r-1)H)
  \end{equation}
  Therefore, we conclude that
  \begin{equation}
    n-r \leq h^0(K_X + rH).
  \end{equation}
  Let $m > r$, and assume that \eqref{eqn:KYmH} holds for $X$ for $m-1$.
  The long exact sequence in cohomology associated to \eqref{eqn:mainexact} gives
  \begin{equation}\label{eqn:add}
    h^0(K_X + (m-1)H) + h^0(K_D + (m-1)H) = h^0(K_X + mH).
  \end{equation}
  By applying the induction hypothesis to $m-1$, we get
  \begin{align*}
    h^0(K_X + (m-1)H) &+ h^0(K_D + (m-1)H)\\
                      &\geq{{m-1} \choose r}(n-r) + {{m-2} \choose {r}} + {{m-1} \choose {r-1}}(n-r) + {{m-2} \choose r-1} \\
                      &={m \choose r} (n-r) + {{m-1} \choose r}.
  \end{align*}
  Together with \eqref{eqn:add}, we conclude 
  \begin{equation}
    {m \choose r} (n-r) + {{m-1} \choose r} \leq h^0(K_X + mH), 
  \end{equation}
  which is \eqref{eqn:KYmH} for $m$.
  The proof of the inequality is thus complete.

  We now examine when equality holds in \eqref{eqn:KYmH}.
  For $r = 1$, the equation \eqref{eqn:r1} shows that equality holds for some $m$ if and only if $g_X = 0$, that is $X \subset \P^n$ is a rational normal curve, and in this case, equality holds for all $m$.
  Furthermore, we observe in the inductive proof that if equality holds for an $X$ of dimension $r > 1$ and some $m$, then it must hold for the hyperplane slice $D$ and $(m-1)$.
  Again, by an induction on $r$, we conclude that $\deg X = n-r+1$, that is, $X \subset \P^n$ is a variety of minimal degree.

  Finally, for $X \subset \P^n$ of minimal degree, induction on $r$ shows that equality holds in \eqref{eqn:KYmH} for all $m$.
\end{proof}

As a consequence, we immediately deduce \autoref{thm:minimaldegree}.
\begin{theorem}[\autoref{thm:minimaldegree}]
  \label{thm:actualminimaldegree}
  Let $X \subset \P^n$ be a smooth, non-degenerate projective variety of dimension $r \geq 1$ over a field of characteristic zero.
  We have the inequality
  \[ \dim \Gr(n-r, n+1) \leq \dim |K_X + (r+1)H|,\]
  where equality holds if and only if $X$ is a variety of minimal degree, that is $\deg X = n-r+1$.
\end{theorem}
\begin{proof}
  Apply \autoref{lem:kymh} with $m = r+1$.
\end{proof}

\subsection{Projection-ramification for scrolls}\label{sec:prscrolls}
\autoref{thm:minimaldegree} motivates a deeper investigation of the projection-ramification map for varieties of minimal degree.
Indeed, for $X \subset \P^n$ of minimal degree, the projection-ramification map is potentially generically finite and dominant.
Recall that a large class of varieties of minimal degree are the rational normal scrolls, namely $X = \P E$ for an ample vector bundle $E$ on $\P^1$ embedded by the complete linear series $\O_X(1)$.
If $\dim X \geq 3$, then $X$ is neither incompressible nor does it have a divisorial dual variety.
Therefore, for such $X$, \autoref{thm:main} leaves the question of maximal variation unanswered.

We now examine the projection-ramification map for projectivizations of vector bundles on smooth curves in more detail.
Let $C$ be a smooth curve and $E$ an ample vector bundle on $C$ of rank $r$.
Set $X = \P E$, the space of one-dimensional quotients of $E$, and $L = \O_X(1)$.
Denote by $\pi \from X \to C$ the natural map.

Let $(L, V)$ be a projection of $X$.
Recall from \eqref{eqn:ramsection} that such a projection gives a map
\[ r_V \from \det V \to H^0(X, K_X \otimes L^{r+1}),\]
whose zero locus is the ramification divisor $R(V) \subset X$.
Note that we have an isomorphism $K_X \cong \pi^* (\det E \otimes K_C) \otimes L^{-r}$, and hence, we may view $r_V$ as a map
\[ r_V \from \det V \to H^0(C, E \otimes \det E \otimes K_C).\]

We now describe another construction of a section of $E \otimes \det E \otimes K_C$ from $V$, which we call the \emph{differential construction}.
The subspace $V \subset H^0(X, L) = H^0(C, E)$ gives the evaluation map
\[ e \from V \otimes \O_C \to E.\]
If $V$ is generic, then $e$ is a surjection, and its kernel is canonically isomorphic to $\det E^* \otimes \det V$.
Consider the diagram
\begin{equation}\label{eqn:differential_construction}
\begin{tikzcd}
  0 \arrow{r}& \det E^* \otimes \det V \arrow{r}\arrow{d}{d_V}& V \otimes \O_C \arrow{r}{e}\arrow{d}{e}& E \arrow{r}\arrow[equal]{d}& 0 \\
  0 \arrow{r}& K_C \otimes E \arrow{r}& P(E) \arrow{r}& E \arrow{r}& 0,
\end{tikzcd}
\end{equation}
where the bottom row is the standard sequence associated to $P(E)$, both maps labeled $e$ are evaluation maps, and the map $d_V$ is the map induced by them.
The map $d_V$ gives a map
\[ d_V \from \det V \to H^0(C, E \otimes \det E \otimes K_C).\]
\begin{proposition}\label{prop:rdv}
  In the setup above, the two maps $d_V$ and $r_V$ are equal.
\end{proposition}
\begin{proof}
  Recall that $r_V$ is induced by the determinant of the evaluation map
  \[ V \otimes \O_X \to P(L).\]
  Denote by $P_\pi(L)$ the bundle of principal parts of $L$ along the fibers of $\pi$.
  More explicitly,
  \[ P_\pi(L) = {\pi_1}_* \left(\pi_2^* L \otimes \left(\O_{X \times_\pi X} / I_{\Delta}^2\right)\right),\]
  where $\Delta \subset X \times_\pi X$ is the diagonal and $\pi_i$ for $i = 1,2$ are the two projections $X \times_\pi X \to X$.
  It is easy to check that the evaluation map $\pi^* E \to L$ induces an isomorphism $\pi^* E \to P_\pi(L)$.
  Furthermore, we have the sequence
  \[ 0 \to \pi^* K_C \otimes L \to P(L) \to P_\pi(L) \to 0.\]
  By combining this with the identification $\pi^* E = P_\pi(L)$, and the top row of \eqref{eqn:differential_construction}, we get the diagram
  \begin{equation}\label{eqn:pxpl}
    \begin{tikzcd}
      0 \arrow{r}& \pi^*(\det E^* \otimes \det V) \arrow{r}\arrow{d}{p}& V \otimes \O_X \arrow{r}\arrow{d}{e}& \pi^* E \arrow{r}\arrow{d}\arrow[equal]{d}& 0\\
      0 \arrow{r}& \pi^* K_C \otimes L \arrow{r}& P(L) \arrow{r}& P_\pi(L) \arrow{r}& 0.
    \end{tikzcd}
  \end{equation}
  From the diagram, we see that $\det e = p$, interpreted as elements of the appropriate $\Hom$ spaces.
  By definition, after taking global sections, $\det e$ gives the section $r_V$.
  Note that, applying $\pi_*$ to the bottom row of \eqref{eqn:pxpl} yields the bottom row of \eqref{eqn:differential_construction}.
  Hence, after applying $\pi_*$, twisting by $\det E$ and taking global sections, $p$ gives the section $d_V$.
  We conclude that $r_V = d_V$.
\end{proof}

Let $R = R(V) \subset X$ be the ramification divisor of the projection given by $V$.
Note that $R$ is a divisor of class $\pi^*(\det E \otimes K_C) \otimes \O_X(1)$.
Therefore, $R \subset X$ is a sub-scroll, or equivalently, the fibers of $R \to C$ are hyperplanes in the corresponding fiber of $X \to C$.
We can obtain an explicit description of these hyperplanes in two ways, one using the original definition, and one using the differential construction.
Fix a point $c \in C$, and a uniformizer $t$ of $C$ at $c$.
Let $X_c \subset X$ and $R_c \subset R$ be the fibers of $X \to C$ and $R \to C$ over $c$, respectively.

By definition $R \subset X$ is the set of points $x \in X$ for which there exists $s \in V$ such that $v(s)$ is singular at $x$.
Since $s$ is a section of $L = \O_X(1)$, the hypersurface $v(s)$ is singular at $x$ if and only if it contains the entire fiber of $\pi \from X \to C$ through $x$.
Suppose $\pi (x) = c$.
Then, in an open set of $X$ containing $X_c$, we have $s = t s_1$ for a section $s_1$ of $\O_X(1)$.
Observe that, we have $\sing(v(s)) \cap F = v(s_1) \cap F$, and therefore, $R_c \subset X_c$ is the hyperplane cut out by $s_1$.

To obtain the same description using the differential construction, consider the top row of \eqref{eqn:differential_construction}.
Let $v$ be a local section of $V \otimes \O_C$ around $c$ that generates the kernel of $e \from V \otimes \O_C \to E$ at $c$.
The fiber of the evaluation map $V \otimes \O_C \to P(L)$ over $c$ sends $v \in V$ to the image of $e(v)$ in $L / \mathfrak m_c^2 L$.
Since $v$ generates the kernel of $e \from V \otimes \O_C \to L$ at $c$, we know that image of $e(v)$ in $L/ \mathfrak m_cL$ is zero.
Writing $e(v) = ts_1$ for a section $s_1$ of $E$ around $c$, we see that $d_V(v) = s_1 \otimes t \in E \otimes \mathfrak m_c/\mathfrak m_c^2$.
Thus, the fiber of the sub-scroll defined by $d_V$ over $c$ is the hyperplane in $X_c$ cut out by $s_1$.

Finally, we write an equation of $R(V) \subset X$ over an open subset of $C$ containing $c$ explicitly in coordinates.
Choose a trivialization $X_1, \dots, X_r$ for $E$ over an open set $U \subset C$ containing $c$.
Then $X_U \cong \P^{r-1} \times U = \Proj \O_U[X_1, \dots, X_r]$.
We have a trivialization of $K_C$ over $U$ given by $dt$.
We then get a trivialization of $P(E)|_U$ by $X_1, \dots, X_r, dt \otimes X_1, \dots, dt \otimes X_r$.
Choose a basis $v_0, \dots, v_r$ of $V$, and suppose the map $e \from V \otimes \O_U \to E_U$ is given by
\[ e(v_i) = \sum m_{i,j} X_j,\]
for $m_{i,j} \in \O_U$, where $0 \leq i \leq r$ and $1 \leq j \leq r$.
Then the map $\det E^* \otimes \det V \to V \otimes \O_U$ defining the kernel of $e$ is given by the $r \times r$ minors of the matrix $(m_{i,j})$.
Denote the $\ell$-th minor by $M_\ell$; that is $M_\ell = (-1)^{\ell}\det (m_{i,j} \mid i \neq \ell)$.
Then the map $d_V$ sends the generator to the element of $E \otimes K_C$ given by
\[ \sum_{i,j} M_i \cdot \frac{\partial m_{i,j}}{\partial t} \cdot (dt \otimes X_j).\]
Note that the expression above is the determinant of the $(r+1) \times (r+1)$ matrix
\begin{equation}\label{eqn:Rmatrix}
  \begin{pmatrix}
  m_{0,1} & m_{0,2} & \dots & m_{0,r} & \sum_{i = 1}^r \frac {\partial m_{0,j}}{\partial t} \cdot dt \otimes X_j \\
  m_{1,1} & m_{1,2} & \dots & m_{1,r} & \sum_{i = 1}^r \frac {\partial m_{1,j}}{\partial t} \cdot dt \otimes X_j \\
  \vdots & \ddots & \dots & \vdots & \vdots \\
  m_{r,1} & m_{r,2} & \dots & m_{r,r} & \sum_{i = 1}^r \frac {\partial m_{r,j}}{\partial t} \cdot dt \otimes X_j \\
\end{pmatrix}.
\end{equation}
This gives an equation for $R_U \subset X_U = \Proj \O_U[X_1, \dots, X_r]$.

\subsection{Failure of maximal variation}\label{sec:failure}
In this section, we show that there exist ample vector bundles $E$ of rank $r \geq 4$ on $\P^1$ such that the projection-ramification map for $X = \P E$ is not generically finite.
In other words, a generic projection of $X$ can be deformed in a one-parameter family so that the ramification divisor remains unchanged.

Recall that the projection-ramification map for $X = \P E$ and the complete linear series of $L = \O_X(1)$ is a map
\[ \rho \from \Gr(r+1, H^0(X, L)) \dashrightarrow |K_X \otimes L^{r+1}|,\]
or equivalently a map
\[ \rho \from \Gr(r+1, H^0(\P^1, E)) \dashrightarrow \P H^0(\P^1, E \otimes \det E \otimes K_{\P^1})^*.\]
By construction, $\rho$ is equivariant with respect to the action of $\Aut(X)$, and in particular, by the subgroup $\Aut(X/\P^1)$.

We engineer the failure of maximal variation using the following observation.
\begin{proposition}\label{prop:trivialStabilizer}
  Let $E$ be an ample vector bundle of rank $r$ on $\P^1$.
  Then a generic point of $\Gr(r+1, H^0(\P^1, E))$ has a trivial stabilizer under the action of $\Aut(\P E/\P^1)$.
\end{proposition}
\begin{proof}
  Fix $(r+1)$ distinct points $p_0, \dots, p_r \in \P^1$.
  Let $V \subset H^0(\P^1, E)$ be a generic $(r+1)$ dimensional subspace.
  Let $e \from V \otimes \O_{\P^1} \to E$ be the evaluation map.
  The points $p_0, \dots, p_r$ give vectors $v_0, \dots, v_r \in V$, unique up to scaling, defined by the property that $e(v_i) = 0$ in the fiber $E|_{p_i}$.
  Choose a generic point $t \in \P^1$.
  We get $(r+1)$ points $x_0, \dots, x_r \in \P E^*|_{t} \cong \P^{r-1}$ given by $e(v_0), \dots, e(v_r)$ evaluated at $t$.
  For generic $V$ and $t$, it is easy to check that these points are in linear general position, using the fact that $E \otimes \O(-1)$ is generated by global sections.
  Any element of $\Aut(\P E/\P^1)$ that fixes $V$ must fix $x_0, \dots, x_r$.
  But then it must act as the identity on the projective space $\P E^*|_t$, and hence on the dual projective space $\P E|_t$.
  Since $t \in \P^1$ is general, it follows that it must be the identity.
\end{proof}

\begin{proposition}\label{prop:specialE}
  There exist ample vector bundles $E$ of every rank $\geq 4$ such that a general point of $\P H^0(\P^1, E \otimes \det E \otimes K_{\P^1})$ has a positive-dimensional stabilizer under $\Aut(\P E/\P^1)$.
  In particular, we may take $E = \O(1)^{r-1} \oplus \O(k+1)$ where $k \geq 1$ and $r \geq 4$.
\end{proposition}
\begin{proof}
  Take
  \[ E = \O(a)^{r-1} \oplus \O(b),\]
  where $0 < a < b$ are to be determined.
  Elements of $\Aut (E/\P^1)$ can be represented by block lower triangular square matrices
  \[M = 
    \begin{pmatrix}
      A &  \\
      U & B
    \end{pmatrix},
  \]
  where $A \in \GL_a(\k)$, $B \in \k^\times$, and $U = (u_i)$ is an $(r-1)$ length row with entries in $H^0(\P^1, \O(b-a))$.
  Set $d = (r-1)a + b$ so that $\det E = \O(d)$.
  Suppose $a$, $b$, and $r$, are such that
  \begin{equation}\label{eqn:requirement}
    (r-1) (b-a+1) \geq b+d-1 = (r-1)a+2b-1.
  \end{equation}
  Take a general element of $H^0(\P^1, E \otimes \det E \otimes K_{\P^1})$; say it is given by the column vector
  \[ v = (p_1, \dots, p_{r-1}, q)^T,\]
  where the $p_i$ (resp $q$) are homogeneous polynomials in $X, Y$ of degree $a+d-2$ (resp $b+d-2$).
  We take $A = \id_{r-1}$ and $B = \lambda$ for some $\lambda \in \k^\times$, and show that there exists a $U = (u_{i})$ such that $Mv = v$.
  Indeed, we have $Mv = (p_1, \dots, p_r, q')$, where
  \[ q' = \lambda q + \sum u_{i}p_i. \]
  Let $W \subset H^0(\P^1, \O(a+d-1))$ be the vector space spanned by $p_1, \dots, p_{r-1}$.
  Consider the multiplication map
  \[ H^0(\P^1, \O(b-a)) \otimes W \to H^0(\P^1, \O(b+d-2)).\]
  Thanks to \eqref{eqn:requirement}, the dimension of the source is at least as much as the dimension of the target.
  It is easy to check that the map is in fact surjective for generic $p_1, \dots, p_{r-1}$.
  In particular, there exist $u_i \in H^0(\P^1, \O(b-a))$ for $i = 1, \dots, r-1$, such that
  \[ q(1-\lambda) = \sum u_i p_i.\]
  With this choice of $U = (u_i)$, we get $M$ such that $Mv = v$.

  Finally, note that the requirement \eqref{eqn:requirement} is satisfied for $a = 1$ and $b = k+1$ if $k \geq 1$ and $r \geq 4$.
\end{proof}

\begin{corollary}[\autoref{thm:counterexamples}]
  \label{cor:actualcounterexamples}
  Let $r \geq 3$ and $d \geq r+1$.
  There exist ample vector bundles $E$ of rank $r$ and degree $d$ on $\P^1$ such that for $X = \P E$ and the complete linear series $L = \O_X(1)$, the projection-ramification map $\rho_X$ is not generically finite onto its image.
\end{corollary}
\begin{proof}
  Take $E$ such that the action of $\Aut(X/\P^1)$ on a generic point of $|K_X\otimes L^{r+1}|$ has a positive-dimensional stabilizer (see \autoref{prop:specialE}).
  Since $\rho_X \from \Gr(r+1, H^0(X, L)) \dashrightarrow |K_X \otimes L^{r+1}|$ is equivariant with respect to the action of $\Aut(X/\P^1)$, and a generic
  point of the source does not have a positive-dimensional stabilizer (see \autoref{prop:trivialStabilizer}), it follows that $\rho_X$ cannot be dominant.
  Since the dimension of the source and target of $\rho_X$ are the same, $\rho_X$ is not generically finite.
\end{proof}

\begin{remark}
In all the examples of scrolls where we know that maximal variation fails, the failure is implied by the presence of generic stabilizers.
We do not know, however, if the presence of stabilizers is {\sl equivalent} to the failure of maximal variation.
\end{remark}

\begin{remark}
  If $k = 1$ and $r \geq 4$, then $X$ is the most balanced scroll of its degree and rank, and hence, generic in moduli.
  Therefore, the non-dominance of projection-ramification is not directly connected to the eccentricity of the splitting type of a scroll. 
\end{remark}

\subsection{Eccentric threefold scrolls} \label{sec:eccentric_threefolds}
\autoref{thm:counterexamples} leaves open the case of threefold scrolls (surface scrolls are covered by \autoref{cor:lowdim}).
We settle this case in this section by showing that the projection-ramification map for threefold scrolls is always generically finite, and thus the statement of \autoref{thm:counterexamples} is sharp in $r$.

Let $E = \O(1) \oplus \O(1) \oplus \O(k+1)$, for $k \geq 0$.
Set $X = \P E$ and $L = \O_X(1)$.
\begin{proposition}\label{prop:eccentric_threefold}
  The map $\rho_X \from \Gr(4, H^0(X, L)) \dashrightarrow |K_X + 4L|$ is birational.
\end{proposition}
\begin{proof}
  The proof is by direct calculation.
  Consider the standard open subset $\A^1 = \spec \k[t] \subset \P^1$.
  Choose trivializations of the three summands of $E$ over $\A^1$ given by sections $X_1, X_2, X_3$.
  
  Let $W \subset H^0(X, L)$ be a general 4-dimensional subspace.
  Then the projection map $W \to H^0(\P^1, \O(1) \oplus \O(1))$ will be an isomorphism.
  Therefore, we can choose a basis of $W$ of the form
  \[
    X_1 + aX_3, X_2 + bX_3, tX_1 + cX_3, tX_2 + d X_3,
  \]
  where $a, b, c, d \in \k[t]$ have degree at most $k+1$.
  Using \eqref{eqn:Rmatrix}, we get that the ramification divisor of this $W$ is
  \begin{equation}\label{eqn:threefoldram}
    \begin{split}
    \rho(W) &= (d-bt)X_1 + (at-c)X_2 + \left((a't-c')(bt-d) + (at-c)(d'-b't) \right) X_3 \\
    &= \alpha X_1 + \beta X_2 + \gamma X_3, \text{ say}.
  \end{split}
\end{equation}
  In this calculation, $p'$ denotes the derivative $\frac{dp}{dt}$.
  Note that we have
  \begin{equation}\label{eqn:alphabetagamma}
    \begin{split}
    \alpha &= d-bt\\
    \beta &= at-c\\
    \gamma &= \alpha'\beta - \beta'\alpha + \alpha a + \beta b.
    \end{split}
  \end{equation}
  The degrees of $\alpha, \beta, \gamma$ are (at most) $k+2$, $k+2$, and $2k+2$, respectively.

  Consider the affine space $\A^{4k+8}$ whose coordinates correspond to the coefficients of $a, b, c, d$, and likewise, the affine space $\A^{4k+9}$ whose coordinates correspond to the coefficients of $\alpha, \beta, \gamma$.
  The expression in \eqref{eqn:threefoldram} defines a map
  \begin{align*}
    \rho^* \from \A^{4k+8} &\to \A^{4k+9} \\
    (a,b,c,d) &\mapsto (\alpha, \beta, \gamma).
  \end{align*}
  Note that the choice of basis of $W$ gives a birational isomorphism $\Gr(4, H^0(X, L)) \cong \A^{4k+8}$.
  Via this isomorphism, the projection-ramification map $\rho$ is simply the composite of $\rho^*$ and the standard projection $\pi \from \A^{4k+9} \setminus \{0\} \to \P^{4k+8} = \P H^0(\P^1, E \otimes \det E \otimes \O(-2))^*$.
  Let $Z \subset \A^{4k+9}$ be the image of $\rho^*$.

  Let $(a,b,c,d) \in \A^{4k+8}$ be a generic point.
  We show that the map induced by $\rho$ on tangent spaces is injective at this point.
  For $\epsilon^2 = 0$, we have
  \[
    \rho^* \from (a + \hat a \epsilon, b+ {\hat b} \epsilon, c +  {\hat c} \epsilon + d +  {\hat d}\epsilon) \mapsto (\alpha + \hat{\alpha} \epsilon,\beta + \hat{\beta} \epsilon, \gamma + \hat{\gamma}\epsilon),
  \]
  where
  \begin{align*}
    \hat \alpha &= \hat d - \hat b t,\\
    \hat \beta &=  \hat a t - \hat c, \text{ and }\\
    \hat \gamma & = (bt-d)({\hat a}'t-{\hat c}') + (at-c)({\hat d}'-{\hat b}'t) \\
    & \qquad + ({\hat a}t-{\hat c})(d'-b't) + ({\hat b}t-{\hat d})(a't-c').
  \end{align*}
  Suppose $\hat \alpha = \hat \beta = \hat \gamma = 0$.
  Then $(bt-d)(\hat a' t - \hat c') + (at-c) (\hat d' - \hat b' t) = 0$.
  However, for generic $a, b, c, d$, the polynomials $(bt-d)$ and $(at-c)$ have degree $(k+2)$ and are relatively prime.
  So they have no non-trivial syzygy with coefficients of degree at most $k+1$.
  As a result, we get $\hat a' t - \hat c' = 0$ and $\hat d' - \hat b' t = 0$.
  Along with $\hat a t - \hat c = 0$ and $\hat d - \hat b t = 0$, we get $\hat a = \hat b = \hat c = \hat d = 0$.
  Thus, $d \rho^*$ is injective, and hence $\rho^* \from \A^{4k+8} \to Z$ is generically finite.

  We now prove that if $(\alpha, \beta, \gamma)$ is a general point in the image of $\rho^{*}$, and $\lambda \neq 0,1$ is a constant, then $\lambda(\alpha, \beta, \gamma)$ is not in the image of $\rho^{*}$.
  In other words, the projection $\pi \from \A^{4k+9} \setminus \{0\} \to \P^{4k+8}$ restricted to $Z \setminus \{0\}$ is generically injective.  
  To this end, suppose $(a,b,c,d)$ is a general point in $\A^{4k+8}$.
  Then  $\alpha = d-bt$ and $\beta = at-c$ will be degree $k+2$ polynomials which are relatively prime.  
  For any polynomial $p(t)$, let $p^{+}$ denote the highest degree coefficient of $p$.
  Observe that $\beta^{+} = a^{+}$.  
  If $\lambda(\alpha, \beta, \gamma)$ is also realized by some quadruple $(\tilde{a}, \tilde{b}, \tilde{c}, \tilde{d})$ then we get the equations: 
	  \begin{align}\label{secondEquation}
	  	\lambda \alpha &= \tilde{d} - \tilde{b}t\\
	  	\lambda \beta &= \tilde{a}t - \tilde{c} \nonumber\\
	  	\lambda \gamma &= \lambda^{2}(\alpha'\beta - \beta' \alpha) + \lambda \alpha \tilde{a} + \lambda \beta \tilde{b}\nonumber
	  \end{align}
	  The second equation gives $\tilde{a}^{+} = \lambda \beta^{+}$.
          The last equation gives
          \[
            \gamma = \lambda(\alpha'\beta - \beta' \alpha) + \alpha \tilde{a} + \beta \tilde{b}.
          \]
          Combining the above with the equation for $\gamma$ in \eqref{eqn:alphabetagamma}, we get 
	  \begin{align*}\label{alphasbetas}
	    	\alpha (a - \beta') + \beta (b + \alpha') &= \alpha(\tilde{a} - \lambda\beta') + \beta(\tilde{b} + \lambda \alpha').
          \end{align*} 
          Since $\alpha$ and $\beta$ are relatively prime and have degree greater than $a,b,\tilde{a},\tilde{b}$, the same syzygy argument gives
          \begin{align*}
            a-\beta' &= \tilde{a} - \lambda \beta'\\
            b+\alpha' &= \tilde{b} + \lambda \alpha'.
          \end{align*} 
	     By examining top coefficients, and using $a^{+} = \beta^{+}$, $\tilde{a}^{+} = \lambda \beta^{+}$ we get
	     \begin{align*}
	     	\beta^{+} - (k+2)\beta^{+} &= \lambda\beta^{+} - \lambda(k+2)\beta^{+}, \text{ or equivalently}\\
               (1-\lambda)\beta^{+} &= (1-\lambda)(k+2)\beta^{+}.
	     \end{align*}
	     Given our assumption on $\lambda$, this is only possible if $\beta^{+} = 0$.  However, since $(a,b,c,d)$ were chosen generically, $\beta^{+} = a^{+}$ would not be zero, providing our desired contradiction.

             We have proved that $\pi \from Z \setminus \{0\} \to \P^{4k+8}$ is of degree 1.
             Therefore, it suffices to show that the degree of $\rho^* \from \A^{4k+8} \to Z$ is 1.
             Since we know that the map is generically finite, it suffices to show that a generic fiber is connected.
             Let $(\alpha,\beta,\gamma) \in Z$ be a generic point.
             The preimage of this point is cut out by the equations in \eqref{eqn:alphabetagamma}.
             Note that these are affine linear equations in $a, b, c, d$, and hence their intersection is an affine space, which is connected.
\end{proof}

\begin{corollary}\label{cor:maxvariation3scrolls}
  The projection-ramification map $\rho_{X}$ is dominant for every smooth three dimensional rational normal scroll $X \subset \P^{n}$.
\end{corollary}
\begin{proof}
  Every such $X$ isotrivially specializes to $\P \left(\O(1) \oplus \O(1) \oplus \O(k+1)\right)$.
  The statement now follows from the upper semi-continuity of fiber dimension.
\end{proof}

The case of eccentric surface scrolls follows by similar calculations as in the proof of \autoref{prop:eccentric_threefold}; we omit the details.
\begin{proposition}\label{prop:eccentric_surface}
  Let $E = \O(1) \oplus \O(k+1)$, for $k \geq 0$.
  Set $X = \P E$ and $L = \O_X(1)$.
  Then the projection-ramification map $\rho_X \from \Gr(3, H^0(X, L)) \dashrightarrow |K_X + 3L|$ is birational.
\end{proposition}

\section{Maximal variation for generic scrolls}\label{sec:generic}
In this section, we establish that the projection-ramification map is generically finite (equivalently, dominant) for most scrolls, notwithstanding the examples provided by \autoref{thm:counterexamples}.
We begin by treating the cases of some particular scrolls by hand.
We then bootstrap these to more general results using degeneration arguments.

\subsection{Maximal variation for some particular cases}\label{sec:lowdegree}

Given an ample vector bundle $E$ on $\P^1$, we say that \emph{maximal variation holds for $E$} if the projection-ramification map is generically finite (equivalently, dominant) for $X = \P E$ embedded by the complete linear series associated to $L = \O_X(1)$.

\begin{proposition}\label{prop:segre}
  Maximal variation holds for $E = \O(1)^r$.
  In fact, the degree of the projection-ramification map in this case is $1$.
\end{proposition}
\begin{proof}
  We know that the projection-ramification map
  \[ \rho \from \Gr(r+1, H^0(\P^1, \O(1)^r)) \dashrightarrow \P H^0(\P^1, \O(r-1)^r)^*\]
  is $\Aut \P E$ equivariant.
  In this case, it is easy to check that the action of $\Aut (\P E / \P^1) = \PGL_r$ has a unique open orbit and trivial generic stabilizers on both the source and the target of $\rho$.
  Hence, $\rho$ must be birational.  
\end{proof}

\begin{proposition}\label{prop:222}
  Maximal variation holds for $E = \O(2)^r$.
\end{proposition}
Compared to \autoref{prop:segre}, our proof of \autoref{prop:222} is significantly more involved, and does not yield the degree.
\begin{proof}
  We exhibit a point $\Gr(r+1, H^0(\P^1, E))$ at which $\rho$ is defined, and at which the induced map $d\rho$ on the tangent space is non-singular.
  It follows that $\rho$ is a local isomorphism at this point, and hence dominant overall.

  Our proof is by direct calculation.
  We calculate on $\A^1 = \spec \k[x] \subset \P^1$ and identify $\O(n)$ with $\O(n \cdot \infty)$.
  Then the global sections of $\O(n)$ are identified with polynomials in $x$ of degree at most $n$.
  Denote the generator of the $i$th summand of $E(-2)$ by $X_i$.
  Consider the point of $\Gr(r+1, H^0(\P^1, E))$ represented by the vector space $V \subset H^0(\P^1, E)$ spanned by the $(r+1)$ sections $v_1, \dots, v_{r+1}$ defined as follows.
  Set $v_i = (x-a_i)^2 X_i$ for $0 \leq i \leq r-1$, and $v_r = \sum p_i X_i$, where $a_i \in \k$, and $p_j \in H^0(\P^1, \O(2))$ are generic.
  By \eqref{eqn:Rmatrix}, the ramification divisor associated to $V$ is cut out by the determinant of the matrix
  \[
    M =
    \begin{pmatrix}
      (x-a_1)^2 & 0 &  \cdots & 0 & 2(x-a_1)X_1\\
      0 & (x-a_2)^2 & \cdots & 0 & 2(x-a_2)X_2\\
      0 & 0 & \ddots & 0 & \vdots\\
      0 & 0 & \cdots & (x-a_r)^2& 2(x-a_r)X_r\\
      p_1 & p_2 & \cdots & p_r & \sum p_i' X_i
    \end{pmatrix}.
  \]
  We leave it to the reader to check that $R = \det M$ is not identically zero.

  To do the tangent space computation, we choose elements $w_i \in H^0(\P^1, E)$, and change $v_i$ to $v_i + \epsilon w_i$, where $\epsilon^2 = 0$.
  Let $R_\epsilon$ be the equation of the discriminant of the projection given by $V_\epsilon \subset H^0(\P^1, E) \otimes \k[\epsilon]/\epsilon^2$, where $V_\epsilon$ is spanned by $v_1 + \epsilon w_1, \dots, v_{r+1} + \epsilon w_{r+1}$.
  Concretely, $R_\epsilon$ is the determinant of a matrix $M_\epsilon$ given by \eqref{eqn:Rmatrix}, which reduces to $M$ modulo $\epsilon$.
  Note that $R_\epsilon$ is an element of $H^0(\P^1, E \otimes \O(2r-2)) \otimes \k[\epsilon]/\epsilon^2$, and we have
  \[ R_\epsilon =  R + \epsilon S(w_1, \dots, w_{r+1}),\]
  for some $S(w_1, \dots, w_{r+1}) \in H^0(\P^1, E \otimes \O(2r-2))$.
  Furthermore, the map
  \begin{equation}\label{eqn:mainmap}
    S \from H^0(\P^1, E)^{r+1} \to H^0(\P^1, E \otimes \O(2r-2))
  \end{equation}
  is a linear map.
  To show that $d \rho$ is non-singular at $V$, it suffices to show that $S$ is surjective.
  For $1 \leq i \leq r$ and $1 \leq j \leq r+1$, let $E_{i,j} \in H^0(\P^1, E)^{r+1}$ be the element corresponding to $(w_1, \dots, w_{r+1})$ where $w_j = X_i$ and $w_\ell = 0$ for all $\ell \neq j$.
  For $i \neq j$ and $1 \leq j \leq r$ and $q \in H^0(\P^1, \O(2))$, by direct calculation we get
  \[ S\left(qE_{i,j}\right) = \frac{(x-a_1)^2 \cdots (x-a_r)^2p_j}{(x-a_i)^2(x-a_j)^2} \cdot [q, (x-a_i)^2] \cdot X_i,\]
  where the notation $[a,b]$ means $a'b-ab'$.
  Similarly, we get
  \[ S\left(qE_{i,r+1}\right) = - \frac{(x-a_1)^2 \cdots (x-a_r)^2}{(x-a_i)^2} \cdot [q, (x-a_i)^2] \cdot X_i,\]
  and
  \begin{equation}\label{eqn:diag}
    S\left(qE_{i,i} \right) = \det M_i,
  \end{equation}
  where $M_i$ is obtained from $M$ by changing the $(i,i)$-th entry from $(x-a_i)^2$ to $q$ and the $(i,r+1)$-th entry from $2(x-a_i)X_i$ to $q'X_i$.

  Fix an $i$ with $1 \leq i \leq r$, and consider the subspace $W_i \subset H^0(\P^1, E)^{r+1}$ spanned by $q E_{i,j}$ for $j \neq i$.
  By our calculations above, $S$ maps $W_i$ to the subspace of $H^0(\P^1, E \otimes \O(2r-2))$ spanned by $H^0(\P^1, \O(2r)) \otimes X_i$.
  We begin by identifying $S(W_i)$.

  For $1 \leq j \leq r$ and $j \neq i$, set
  \[
    Q_{i,j} = \frac{(x-a_1)^2 \cdots (x-a_r)^2p_j}{(x-a_i)^2(x-a_j)^2}, 
  \]
  and
  \[
    Q_{i,r+1} = - \frac{(x-a_1)^2 \cdots (x-a_r)^2}{(x-a_i)^2}.
  \]
  We claim that, there is no non-trivial linear relation among the $r$ polynomials $Q_{i,j}$ for $j \in \{1, \dots, r+1\} \setminus \{i\}$.
  Indeed, suppose we had a linear relation
  \[ \sum l_j Q_{i,j} = 0,\]
  then dividing throughout by $\frac{(x-a_1)^2\cdots (x-a_r)^2}{(x-a_i)^2}$ gives the relation
  \[ \sum_{j = 1}^r l_j \frac{p_j}{(x-a_j)^2} + l_{r+1} = 0.\]
  If $l_j \neq 0$ for some $j$ with $1 \leq j \leq r$, then we have a pole on the left side at $x = a_j$, but not on the right side (note that $(x-a_j)$ does not divide $p_j$ by the genericity of $p_j$).
  Therefore, we must have $l_j = 0$ for all $j$, and hence also $l_{r+1} = 0$.
  Consider the map
  \begin{equation}\label{eqn:big}
    H^0(\P^1, \O(1)) \otimes \langle  Q_{i,j} \mid j \in \{1, \dots, r+1\} \setminus \{i\}\rangle \to H^0(\P^1, \O(2r-1)).
  \end{equation}
  We just saw that this map is injective.
  But both sides have the same dimension, and hence the map must be surjective.
  Finally, it is easy to see that the image of the map
  \begin{equation}\label{eqn:q}
    H^0(\P^1, \O(2)) \to H^0(\P^1, \O(2)), \quad q \mapsto [q, (x-a_i)^2]
  \end{equation}
  is $(x-a_i)\cdot H^0(\P^1, \O(1))$.
  By \eqref{eqn:big} and \eqref{eqn:q}, we conclude that the image of the map
  \[ S \from W_i = \langle qE_{i,j} \mid j \in\{1, \dots, r+1\} \setminus \{i\} \to H^0(\P^1, \O(2r)) \otimes X_i\]
  is $(x-a_i)H^0(\P^1, \O(2r-2)) \otimes X_i$.
  In other words, the cokernel of the map is $\k \otimes X_i$ where the map
  \[H^0(\P^1, \O(2r)) \otimes X_i \to \k \otimes X_i \]
  is given by evaluation at $a_i$.
  Putting together the maps for various $i$, we see that the cokernel of the map
  \[ S \from \bigoplus_i W_i \to H^0(\P^1, E \otimes \O(2r-2)) = H^0(\P^1, \O(2r)) \otimes \langle  X_1, \dots, X_r \rangle\]
  is $\k \otimes \langle  X_1, \dots, X_r \rangle$, where the map
  \begin{equation}\label{eqn:partialsur}
    H^0(\P^1, E \otimes \O(2r-2)) = H^0(\P^1, \O(2r)) \otimes \langle  X_1, \dots, X_r \rangle \to \k \otimes \langle  X_1, \dots, X_r \rangle
  \end{equation}
  on $H^0(\P^1, \O(2r)) \otimes X_i$ is given by evaluation at $a_i$.

  To show that $S$ is surjective, it is now enough to show that the map
  \begin{equation}\label{eqn:remainsur}
    H^0(\P^1, \O(2)) \otimes \langle  qE_{i,i} \mid i \in \{1, \dots, r+1\} \rangle \to \k \otimes \langle  X_1, \dots, X_r \rangle
  \end{equation}
  obtained by composing \eqref{eqn:mainmap} and \eqref{eqn:partialsur} is surjective.
  Recall from \eqref{eqn:diag} that we have $S(qE_{i,i}) = \det M_i$, where $M_i$ is obtained from $M$ by changing the $(i,i)$-th entry to $q$ and the $(i, r+1)$-th entry to $q'X_i$.
  Taking $q = (x-a_i)$ gives
  \[ S(qE_{i,i}) = \det M_i = \pm \prod_{j \neq i} (a_i-a_j)^2 p_i(a_i) X_i,\]
  which is a non-zero multiple of $X_i$.
  That is, the images of $(x-a_i)E_{i,i}$ under $S$ span $\k \otimes \langle  X_1, \dots, X_r \rangle$, and hence the map in \eqref{eqn:remainsur} is surjective.
  The proof is now complete.
\end{proof}

Our next goal is to bootstrap from \autoref{prop:segre} and \autoref{prop:222} to deduce maximal variation for generic scrolls of sufficiently high degree.
We do this by a degeneration argument.
We degenerate a vector bundle $E$ to a vector bundle $E_0$ on the nodal rational curve $P_0 = \P^1 \cup \P^1$, and show that the projection-ramification map for $E_0$ is dominant.
For this to work, we have to define the projection-ramification map for nodal curves.
It turns out that with the most na\"ive definition of linear series on scrolls on nodal curves, we do not get a dominant projection-ramification map.
As a remedy, we work with the (linked) limit linear series of higher rank as developed in \cite{tei-i-big:91} and \cite{oss:14}.
In the literature, there are a few different versions of the notion of a limit linear series.
We use \cite{oss:14} for the foundations of the theory, and following the terminology there, call our limit linear series linked linear series.

\subsection{Linked linear series}\label{sec:lls}
We need linked linear series for the simplest singular curve, namely a (projective, connected) nodal curve $C$ which is the nodal union of two smooth (projective, connected) curves $C_1$ and $C_2$, but we need them for vector bundles of rank higher than $1$.
Let $B$ be the spectrum of a DVR with special point $0$, general point $\eta$.
Let $\pi \from X \to B$ be a smoothing of $C$ with non-singular total space $X$.
That is, $\pi$ is a flat, proper, family of connected curves, smooth over $\eta$, and isomorphic to $C$ over $0$.
Such a family is a particularly simple example of an almost local smoothing family \cite[\S~2.1--2.2]{oss:14}.
Let $g_i$ be the genus of $C_i$ for $i = 1, 2$, and $g = g_1+g_2$ the genus of $X_\eta$.

Let $E$ be a vector bundle of rank $r$ on $C$.
The \emph{multi-degree} of $E$ is the pair of integers $(\deg E|_{C_1}, \deg E|_{C_2})$.
The \emph{degree} or \emph{total degree} of $E$ is the sum $\deg E = \deg E|_{C_1} + \deg E|_{C_2}$.

Once and for all, fix a vector bundle $\mathcal E$ of rank $r$ on $X$, and set $E = \mathcal E|_C$.
Let $E$ have degree $d$ and multi-degree $(w_1, w_2)$.
Fix a positive integer $k$.
Our next task is to recall the definition of the space of linked linear series of dimension $k$.
It will be a $B$-scheme whose fiber over $\eta$ is the Grassmannian $\Gr(k, H^0(X_\eta, \mathcal E_\eta))$.
The key idea is to not only consider the sections of $\mathcal E$, but also of its various twists, namely the vector bundles obtained by tensoring with the powers of $\O_X(C_i)$.

Fix maps $\theta_1 \from \O_X \to \O_X(C_1)$ and $\theta_2 \from \O_X \to \O_X(C_2)$.
The choice of these maps is auxiliary, and each one is unique up to multiplication by an element of $\O_B^*$.
For $n \in \Z$, set
\[ \mathcal E_n =
  \begin{cases}
    \mathcal E \otimes \O_X(C_1)^{\otimes n} & \text{if $n \geq 0$},\\
    \mathcal E \otimes \O_X(C_2)^{\otimes (-n)}  & \text{if $n < 0$}.
  \end{cases}
\]
The maps $\theta_1$ and $\theta_2$ induces maps
\[ \theta_n \from \mathcal E_m \to \mathcal E_{m+n}\]
given by
\[
  \theta_n = 
  \begin{cases}
    \theta_1^n & \text{if $n \geq 0$,} \\
    \theta_2^{-n} & \text{if $n < 0$.}
  \end{cases}
\]
Note that the multi-degree of $\mathcal E_n$ is $(w_1 - nr, w_2 + nr)$.
In particular, for sufficiently negative $n$, say for $n \leq n_1$, we have $H^0(C_2, \mathcal E_n|_{C_2}) = 0$, and similarly, for sufficiently positive $n$, say $n \geq n_2$, we have $H^0(C_1, \mathcal E_n|_{C_1}) = 0$.
Assume, without loss of generality, that $n_2 \geq n_1$.
Set
\[ d_1 = w_1 - n_1 r, \text{ and } d_2 = w_2 + n_2 r, \text{ and } b = n_2 - n_1.\]
Observe that
\[ d_1 + d_2 - rb = d.\]

\begin{definition}[linked linear series]
  \label{def:lls}
  Let $S$ be a $B$-scheme.
  A \emph{$k$-dimensional linked linear series} on $\mathcal E_S$ consists of sub-bundles $V_n \to \pi_* (\mathcal E_n)_S$ of rank $k$ for every $n \in \Z$ satisfying the following compatibility condition.
  For every $m, n \in \Z$, the map
  \begin{equation}\label{lls:compatibility}
    \pi_* \theta_n \from \pi_* (\mathcal E_m)_S \to \pi_* (\mathcal E_{m+n})_S \text{ maps } V_m \to V_{m+n}.
  \end{equation}
\end{definition}
\autoref{def:lls} is a special case of \cite[Definition~3.3.2]{oss:14}.
From now on, we will talk about the image of an element in $V_m$ in $V_{m+n}$; this should be understood as the image under the map $\pi_* \theta_n$.

\begin{remark}
The notion of a sub-bundle of a push-forward is a bit subtle; it is treated in depth in \cite[Definition~B.2.1]{oss:14}.
We recall the main points.
For a flat proper morphism $X \to S$ and a vector bundle $\mathcal E$ on $S$, a \emph{sub-bundle} of $\pi_* \mathcal E$ is a vector bundle $V$ on $S$ along with a map $i \from V \to \pi_* \mathcal E$ such that for every $T \to S$, the pull-back $i_T \from V_T \to \pi_* (\mathcal E_T)$ is injective.
Note that this is a local condition on $S$.
For Noetherian schemes such as ours, it is enough to check this condition for the $T \to S$ that are inclusions of closed points.
Alternatively, if $F_0 \to F_1 \to \cdots $ is a complex of vector bundles on $S$ quasi-isomorphic to $R\pi_* \mathcal E$, then a sub-bundle of $\pi_* \mathcal E$ is a vector bundle $V$ along with a map $i \from V \to \pi_* \mathcal E$ such that the composite $V \to F_0$ is an injection of vector bundles (that is, the dual map is surjective).
\end{remark}

\begin{remark}
\autoref{def:lls} defines linked linear series on a particular vector bundle $\mathcal E$.
We can also vary the choice of the vector bundle, as is done in \cite{oss:14}; in that case, one imposes an additional vanishing condition on the vector bundles to ensure boundedness of the moduli space of linked linear series.
\end{remark}

\begin{definition}[Simple linked linear series]
  \label{def:simple_lls}
Let $S = \spec K$, where $K$ is a field, and let $V = (V_n \mid n \in \Z)$ be a linked linear series on $S$.
We say  $V$ is \emph{simple} if there exist integers $w_1, \dots, w_k$, not necessarily distinct, and elements $v_i \in V_{w_i}$ such that for every $w \in \Z$, the images of $v_1, \dots, v_k$ in $V_w$ form a basis of $V_w$.
\end{definition}

Note that if $S \to B$ maps to the generic point $\eta$, then the data of a linked linear series $V = (V_n)$ is equivalent to the data of an individual $V_n$ for any $n \in \Z$, and in particular, for $n = 0$.
As a result, the functor that associates to $S \to \eta$ the set of $k$-dimensional linked linear series of $\mathcal E_S$ is represented by the Grassmannian $\Gr(k, H^0(X_\eta, \mathcal E_\eta))$.
The main theorem of \cite{oss:14} is the following representability theorem.
\begin{theorem}[{\cite[Theorem~3.4.7]{oss:14}}]
  \label{thm:lls}
  The functor that associates to a $B$-scheme $S \to B$ the set of linked linear series on $\mathcal E_S$ is representable by a projective $B$-scheme $\mathcal G(k, \mathcal E)$ isomorphic to the Grassmannian $\Gr(k, H^0(X_\eta, \mathcal E_\eta))$ over $\eta$.
  The locus of simple linear series ${\mathcal G}^{\rm simple}(k, \mathcal E) \subset {\mathcal G}(k, \mathcal E)$ is an open subscheme, and the map ${\mathcal G}^{\rm simple}(k, \mathcal E) \to B$ has universal relative dimension at least $k(d-k-r(g-1))$.
\end{theorem}
The last statement implies that if $v \in {\mathcal G}^{\rm simple}$ is such that ${\mathcal G}^{\rm simple}$ has relative dimension at most $k(d-k-r(g-1))$ at $v$, then it has relative dimension exactly $k(d-k-r(g-1))$ at $v$ and, furthermore, it is an open map near $v$.
In particular, $v$ is in the closure of $\Gr(k, H^0(X_\eta, \mathcal E_\eta))$.
\begin{remark}
  Osserman proves a stronger theorem, namely a relative version of the statement above, over the stack of vector bundles on $X$.
  But the statement above is enough for our purposes.  
\end{remark}

Although the definition of a linked linear series demands that we specify infinitely many vector bundles $V_n$, one for each $n \in \Z$, this is neither practical nor necessary.
In the best case, only specifying the extremal ones, namely $V_{n_1}$ and $V_{n_2}$, suffices, provided that they satisfy some compatibility conditions.
The original definition of limit linear series due to Eisenbud--Harris \cite{eis.har:86,eis.har:84} in the rank 1 case and Teixidor i Bigas \cite{tei-i-big:91} in the general case, took this minimalist approach.

Let $E_n$ be the restriction of $\mathcal E_n$ to the central fiber $C = X_0$, and set $p = C_1 \cap C_2$.
\begin{definition}[EHT limit linear series]
  \label{def:eht}
  A \emph{$k$-dimensional EHT limit linear series} on $E$ consists of $k$-dimensional subspaces $W_i \subset H^0(C_i, E_{n_i}|_{C_i})$ for $i = 1, 2$ that satisfy the following two conditions.
  \begin{enumerate}
  \item
    \label{ieq:eht}
    If $a^i_1 \leq \cdots \leq a^i_k$ is the vanishing sequence for $(\mathcal E_{n_i}|_{C_i}, W_i)$ at $p$ for $i = 1, 2$, then for every $v = 1, \dots, k$ we have
    \[ a^1_v + a^2_{k+1-v} \geq b.\]
  \item\label{gluing:eht}
    There exist bases $s^i_1, \dots, s^i_k$ for $W_i$ for $i = 1, 2$, such that $s^i_v$ has order of vanishing $a^i_v$ at $p$, and if we have $a^1_v + a^2_{k+1-v} = b$ for some $v$, then
    \[ \widetilde \phi (s^1_v) = s^2_{k+1-v},\]
    where $\widetilde \phi \from E_{n_1}(-a^1_{v}\cdot p)|_p \to E_{n_2}(-a^2_{k+1-v} \cdot p) |_p$ is the isomorphism obtained by taking the appropriate twist of the identity map.
  \end{enumerate}
  We say that $(W_1, W_2)$ is a \emph{refined EHT limit linear series} if equality holds in \eqref{ieq:eht} for all $v = 1, \dots, k$.
\end{definition}
This definition is adapted from \cite[Definition~4.1.2]{oss:14}.
Note that, due to the vanishing condition on the twists of $E$, the restriction map
\[ H^0(C, E_{n_i}) \to H^0(C_i, E_{n_i}|_{C_i})\]
is an injection.
Via this injection, we sometimes treat $W_i$ as a subspace of $H^0(C_i, E_{n_i}|_{C_i})$.

Although the notions of a linked linear series and an EHT limit linear series differ in general, they essentially agree when we restrict to the simple linked linear series and the refined EHT limit linear series.
More precisely, we have the following statement.
\begin{proposition}\label{prop:llseht}
  Let $S$ be a $B$-scheme, and $V = (V_n \mid n \in \Z)$ a linked linear series on $\mathcal E_S$.
  For every $s \in S$ over $0 \in B$, taking $W_i = V_{n_i}|_s$ for $i = 1, 2$ gives an EHT limit linear series.
  Conversely, assume that $S$ reduced, and let $\mathcal W_i \subset \pi_*(\mathcal E_{n_i})_S$ for $i = 1,2$ be sub-bundles whose restrictions to every $s \in S$ over $\eta \in B$ agree under the isomorphism $(\mathcal E_{n_1})_\eta \cong (\mathcal E_{n_2})_\eta$, and to every $s \in S$ over $0 \in B$ define a refined EHT limit linear series.
  Then there exists a unique linked linear series $V = (V_n \mid n \in \Z)$ on $\mathcal E_S$ such that $\mathcal W_i = V_{n_i}$.
  Furthermore, for every $s \in S$ over $0$, the series $V|_s$ is simple.
\end{proposition}
\begin{proof}
  Proving that $(W_1, W_2)$ is an EHT limit linear series is straightforward, and left to the reader.
  It is a special case of \cite[Theorem~4.3.4]{oss:14} and the equivalence of type I and type II series in the two component case (\cite[Remark~3.4.15]{oss:14}.

  The converse also follows from the proof of \cite[Theorem~4.3.4]{oss:14}, but it is not explicitly stated there.
  So we offer a proof.
  
  First, suppose that $S$ lies over $\eta \in B$.
  Then $V_n \subset \pi_* (\mathcal E_n)_S$ is determined uniquely as the image of $V_{n_i} = \mathcal W_{n_i} \subset \pi_* (\mathcal E_{n_i})_S$ for either $i = 1$ or $i = 2$.

  Next, suppose that $S = \spec K$, and it lies over $0 \in B$.
  Denoting $(\mathcal E_n)_S$ by $E_n$, we must construct $V_n \subset H^0(C, E_n)$.
  By composing $\theta_{n_i-n} \from E_n \to E_{n_i}$ and the restriction $E_{n_i} \to E_{n_i}|_{C_i}$, we get a map
  \[ \iota \from H^0(C, E_n) \to H^0(C_1, E_{n_1}|_{C_1}) \oplus H^0(C_2, E_{n_2}|_{C_2}). \]
  The vanishing condition on the twists of $E$ mean that $\iota$ is injective.
  The compatibility condition in \autoref{def:lls} implies that we must choose $V_n$ so that $\iota (V_n) \subset W_1 \oplus W_2$.
  We claim that $\dim \iota^{-1}(W_1 \oplus W_2) = k$, so that there is a unique choice of $V_n$, namely $V_n = \iota^{-1}(W_1 \oplus W_2)$.

  Suppose $s \in \iota^{-1}(W_1 \oplus W_2)$.
  Then $\iota(s)$ is a linear combination of $(s^1_1,0), \dots, (s^1_k, 0)$, and $(0, s^2_1), \dots, (0,s^2_k)$.
  Write $\iota(s) = (s_1, s_2)$.
  Since $s_i$ is obtained by applying $\theta_{n-n_i}$, and $\theta$ on $C_i$ at $p$ corresponds to multiplication by the uniformizer, we see that
  \begin{equation}\label{eqn:vanishing}
    \ord_p(s_1) \geq n - n_1, \text{ and likewise, } \ord_p(s_2) \geq n_2 - n.
  \end{equation}
  Let $v_1 \in \{1, \dots, k\}$ be the smallest such that $a^1_{v_1} \geq n-n_1$, and $v_1+c$ the smallest such that $a^1_{v_1+c} > n-n_1$.
  Since $(W_1, W_2)$ is refined, and $n_2 - n_1 = b$, we see that $v_2 = k+1-v_1$ is the largest such that $a^2_{v_2} \leq n_2 - n$, and $v_2 - c$ the smallest such that $a^2_{v_2 + c} < n_2 - n$.
  The vanishing conditions \eqref{eqn:vanishing} imply that $\iota(s)$ must be a linear combination of $(s^1_{v_1},0), \dots, (s^1_k,0)$ and $(0, s^2_{v_2-c}), \dots,  (0,s^2_k)$.
  Suppose
  \[ \iota(s) = \sum_{\ell = v_1}^k \alpha_{\ell} \cdot (s^1_\ell,0) + \sum_{\ell = v_2-c}^{k} \beta_\ell \cdot (0,s^2_\ell),\]
  where $\alpha_{\ell}$ and $\beta_{\ell}$ are elements of the field $K$.
  Since $s$ is a section on the entire nodal curve $C$, its two restrictions to $C_1$ and $C_2$ are equal at $p$.
  In terms of the two components of $\iota(s)$, and in light of the gluing condition \eqref{gluing:eht} in \autoref{def:eht}, this equality is equivalent to $\alpha_{\ell} = \beta_{k+1-\ell}$ for $v_1 \leq \ell < v_1+c$.
  That is, $\iota(s)$ is a linear combination of the $k$ elements 
  \[ (s^1_{v_1} , s^2_{v_2}), \dots, (s^1_{v_1+c-1} , s^2_{v_2-c+1}), (s^1_{v_1+c},0), \dots, (s^1_k,0), (0,s^2_{v_2+1}), \dots, (s^2_{k},0).\]  
  Conversely, it is easy to see that any such linear combination lies in $W_1 \oplus W_2$.
  Hence the claim that $\dim \iota^{-1}(W_1 \oplus W_2) = k$.

  Set $V_n = \iota^{-1}(W_1 \oplus W_2)$.
  To see that $V$ is simple, we must exhibit appropriate $w_i$ and $v_i \in V_{w_i}$ for $i = 1, \dots, k$.
  Take $w_i = n-n_1-a^1_i$, and let $v_i \in V_{w_i} \subset H^0(C, E_{w_i})$ be such that $\iota(v_i) = (s^1_i, s^2_{k+1-i})$.
  Then the images of $v_1, \dots, v_k$ form a basis of $V_n$ for all $n \in \Z$.

  For more general $S$, consider the map
  \[ \overline \iota \from \pi_* (\mathcal E_n)_S \to \pi_*(\mathcal E_{n_1})_S / \mathcal W_1 \oplus \pi_*(\mathcal E_{n_2})_S / \mathcal W_2,\]
  obtained by composing $\iota = \pi_*(\theta_{n_1-n} \oplus \theta_{n_2-n})$ and the projections $\pi_*(\mathcal E_{n_i})_S \to \pi_*(\mathcal E_{n_i})_S / \mathcal W_i$.
  We proved that, for every $\spec K \to S$, the kernel of $\overline \iota \otimes_{\O_S} K$ is $k$-dimensional.
  Since $S$ is reduced, it is easy to prove that $V_n = \ker \iota$ is a sub-bundle of $\pi_*(\mathcal E_n)$ (see \cite[B.3.4 with reduced $B$]{oss:14}).
  It is also easy to check that $V = (V_n \mid n \in \Z)$ is a linked linear series, the only one that satisfies $V_{n_i} = \mathcal W_i$.
  The proof is now complete.
\end{proof}

\autoref{prop:llseht} allows us to combine the economy of specifying an EHT limit linear series with the convenient functorial definition of a linked linear series.
We use this in the definition of the projection-ramification map in terms of linked linear series.

\subsection{Projection-ramification with non-generic vanishing sequence}
\label{sec:prnongeneric}
We consider the projection-ramification map for linear series with a non-generic vanishing sequence.
The analysis of such series plays a key role in defining the projection-ramification map for linked linear series.

Let $C$ be a smooth curve and $p \in C$ a point.
Let $E$ be a vector bundle on $C$ of rank $r$.
The projective spaces associated to the vector spaces $E(np)|_p$, for $n \in \Z$, are canonically isomorphic to each other, so we identify them.
The vanishing sequences considered are at the point $p$.
Choose a uniformizer $t$ of $C$ at $p$.

Suppose $V \subset H^0(C, E)$ is an $(r+1)$-dimensional subspace with the vanishing sequence 
\begin{equation}\label{eqn:specialvs}
  (\underbrace{a, \dots, a}_{i}, \underbrace{a+1, \dots, a+1}_{r+1-i}),
\end{equation}
for some $i$ with $1 \leq i \leq r$, and $a \geq 0$.
Let $v_1, \dots, v_{r+1}$ be a basis of $V$ adapted to the vanishing sequence, namely a basis $v_1, \dots, v_{r+1}$ such that in the stalk $E_p$, we can write
\begin{equation}\label{eqn:basis}
  v_1 = t^a \widetilde v_1, \dots, v_{i} = t^a \widetilde v_i,\quad v_{i+1} = t^{a+1} \widetilde v_{i+1}, \dots, v_{r+1} = t^{a+1} \widetilde v_{r+1},
\end{equation}
for some $\widetilde v_1, \dots, \widetilde v_{r+1} \in E_p$ such that the images of $\widetilde v_1, \dots, \widetilde v_i$ in the fiber $E|_p$ are linearly independent, and the same holds for the images of $\widetilde v_{i+1}, \dots, \widetilde v_{r+1}$.
Here we are slightly abusing the notation by denoting $v_i$ and its image in $E_p$ under the natural evaluation map by the same letter.
Let $V^0 \subset E|_p$ be spanned by the images of $\widetilde v_1, \dots, \widetilde v_i$, and $V^1 \subset E|_p$ by the images of $\widetilde v_{i+1}, \dots, \widetilde v_{r+1}$.
It is easy to check that a different choice of basis adapted to the vanishing sequence gives the same $V^0$ and $V^1$.
By construction, $\dim V_0 = i$ and $\dim V^1 = r+1-i$, and therefore, $\dim (V^0 \cap V^1) \geq 1$.
We say that $V$ has \emph{transverse vanishing} at $p$ if 
\begin{equation}\label{eq:genericity}
  \dim (V^0 \cap V^1) = 1.
\end{equation}
Note that if $V$ is base-point free at $p$, then $\dim V^0 = r$ and $\dim V^1 = 1$, so $V$ automatically has transverse vanishing.

\begin{proposition}\label{prop:agreement}
  Suppose $V \subset H^0(C, E)$ is an $(r+1)$-dimensional subspace with vanishing sequence \eqref{eqn:specialvs} and transverse vanishing at $p$.
  Then the ramification section $r_V$ of $V$ vanishes to order $(r+1)a + (r-i)$ at $p$.
  Furthermore, writing $r_V = t^{(r+1)a+r-i} \cdot \widetilde r$, the one-dimensional subspace of $E|_p$ spanned by $\widetilde r|_p$ is $V^0 \cap V^1$.
\end{proposition}
\begin{proof}
  Thanks to transverse vanishing, there exists a basis $\{\overline s_1, \dots, \overline s_{r} \}$ of $E|_p$ such that
  \[ V^0 = \langle  \overline s_1, \dots, \overline s_i \rangle \text{ and } V^1 = \langle  \overline s_{i+1}, \dots, \overline s_r, \overline s_1 \rangle.\]
  Let $v_1, \dots, v_{r+1}$ be a basis of $V$ adapted to the vanishing sequence such that if $\widetilde v_i$ are defined as in \eqref{eqn:basis} then the images of $\widetilde v_1, \dots, \widetilde v_r$ in $E|_p$ are $\overline s_1, \dots, \overline s_r$, respectively, and the image of $\widetilde v_{r+1}$ is $\overline s_1$.
  In particular, the $r$ elements $\widetilde v_1, \dots, \widetilde v_r \in E_p$ give a trivialization of $E$ around $p$.
  Write
  \[ \widetilde v_{r+1} = b_1 \widetilde v_1 + \dots + b_r \widetilde v_r\]
  in $E_p$, where $b_1, \dots, b_r \in \O_{C,p}$.
  Since the image of $\widetilde v_{r+1}$ in $E|_p$ is $\overline s_1$, we get that $b_1 \equiv 1 \pmod {\mathfrak m_p}$, and $b_2, \dots, b_r \in \mathfrak m_p$.  
  Using the basis $v_1, \dots, v_{r+1}$ of $V$ and the local trivialization $\widetilde v_1, \dots, \widetilde v_r$ of $E$, we can write $r_V$ as the determinant (see \eqref{eqn:Rmatrix}) as follows
  \begin{align*}
    r_V &= \det
    \begin{pmatrix}
      t^a & & & & & &  at^{a-1}\widetilde v_1\\
       & \ddots & & & & &\vdots\\
       & & t^a  & & & &a t^{a-1}\widetilde v_i\\
       & & & t^{a+1}  & & &(a+1)t^a \widetilde v_{i+1} \\
       & & & & \ddots & & \vdots\\
       & & & & &t^{a+1} &(a+1)t^{a}\widetilde v_r\\
      b_1t^{a+1}& b_2 t^{a+1} & \cdots & b_{r-1}t^{a+1} & & b_rt^{a+1} & (a+1)t^a\widetilde v_1 + t^{a+1}(\cdots)
    \end{pmatrix}\\
        &= t^{(r+1)a+r-i} \widetilde v_1  + t^{(r+1)a+r-i+1} (\cdots).
  \end{align*}
  Thus the order of vanishing of $r_V$ is as claimed.
  Furthermore, $\widetilde r$ is given by
  \[ \widetilde r = \widetilde v_1 + t (\cdots).\]
  Since the image of $\widetilde v_1$, namely $\overline s_1$, spans $V^0 \cap V^1$, the proof is complete.
\end{proof}

We are primarily interested in generic $(r+1)$-dimensional subspaces $V \subset H^0(C, E)$.
A generic such $V$ has the vanishing sequence
\[ (0, \dots, 0, 1).\]
For linked linear series, it is important to also study the $V$ with complementary vanishing sequence, namely
\[ (0,1, \dots, 1),\]
which we now do.
For simplicity, we restrict to $C = \P^1$.

Let $E$ be an ample vector bundle on $\P^1$ of rank $r$.
Fix a point $p \in \P^1$; all the vanishing sequences are at $p$.
Consider  the locally closed subset $U \subset \Gr(r+1, H^0(\P^1, E))$ parametrizing $V \subset H^0(\P^1, E)$ with vanishing sequence
\[ (0,\underbrace{1,\dots, 1}_{r}).\]
Given such a $V$, let $\widetilde r_V \in \P H^0(E \otimes \det E \otimes K_{\P^1} \otimes \O(-(r-1)p)^*$ be the reduced ramification section, namely the section obtained by dividing the usual ramification section $r_V$ by the $(r-1)$-th power of a uniformizer at $p$ (see \autoref{prop:agreement}).
The assignment $V \mapsto \widetilde r_V$ gives a variant of the projection-ramification map, which we call the \emph{reduced projection-ramification map}
\begin{equation}\label{eqn:rrd}
  \widetilde \rho \from U \to \P H^0(\P^1, E \otimes \det E \otimes K_{\P^1} \otimes \O(-(r-1)p))^*.
\end{equation}
Note that, just as in the case of the usual projection-ramification map, the source and the target of the reduced projection-ramification map are of the same dimension.

Having defined the reduced projection-ramification map, we now relate it back to the usual projection-ramification map, but on a different vector bundle.
Given a one-dimensional subspace $\ell \subset E|_p$, define $E'_\ell$ by the exact sequence
\[ 0 \to E_\ell' \to E \to E|_p/\ell\to 0.\]
There exists a Zariski open subset of the projective space of lines in $E|_p$ such that for all $\ell$ in this set, the isomorphism class of $E'_{\ell}$ remains constant.
Denote this isomorphism class by $E'_{\rm gen}$.
\begin{proposition}\label{prop:domred}
  If the usual projection-ramification map
  \[ \rho \from \Gr(r+1, H^0(\P^1, E'_{\rm gen})) \dashrightarrow \P H^0(\P^1, E'_{\rm gen} \otimes \det E'_{\rm gen} \otimes K_{\P^1})^*\]
  is dominant, then so is the reduced projection-ramification map
  \[\widetilde \rho \from U \to \P H^0(\P^1, E \otimes \det E \otimes K_{\P^1} \otimes \O(-(r-1)p))^*.\]
\end{proposition}
\begin{proof}
  Let $D \in \P H^0(E \otimes \det E \otimes K_{\P^1} \otimes \O(-(r-1)p))^*$ be a generic section.
  Let $\ell \subset E|_p$ be the one-dimensional subspace defined by $D|_p$, and set $E' = E'_{\ell}$.
  Since $D$ is generic, we may assume $E' \cong E'_{\rm gen}$.
  The inclusion of sheaves $E' \to E$ induces an inclusion of sheaves
  \[
    E' \otimes \det E' \otimes K_{\P^1} \to E \otimes \det E \otimes \O(-(r-1)p) \otimes K_{\P^1},
  \]
  and by construction, $D$ is the image of a section $D' \in \P H^0(E' \otimes \det E' \otimes K_{\P^1})^*$.
  Since $\rho$ is dominant for $E'$, there exists a sequence of subspaces $V_n' \in \Gr(r+1, H^0(\P^1, E'))$ such that the limit of $\rho(V_n')$ is $D'$.
  Let $V_n \subset \Gr(r+1, H^0(\P^1, E))$ be the image of $V'_n$.
  Then the limit of $\widetilde \rho(V_n)$ is $D$.
  Since $D$ was generic, we get that $\widetilde \rho$ is dominant.
\end{proof}

\begin{corollary}\label{prop:domredexamples}
  The reduced projection-ramification map is dominant for the bundles $E = \O(1) \oplus \O(2)^{r-1}$ and $E = \O(2) \oplus \O(3)^{r-1}$.
\end{corollary}
\begin{proof}
  Follows from \autoref{prop:domred} and that the projection-ramification map is dominant for $E' = \O(1)^r$ and $E' = \O(2)^r$.
\end{proof}

\subsection{Projection-ramification for linked linear series}
\label{sec:prlls}
Recall the setup from \autoref{sec:lls}: $C = C_1 \cup C_2$ is a nodal union of two smooth projective curves of genus $g_1$ and $g_2$, and $\pi \from X \to B$ be a smoothing of $C$.
Let $\mathcal E$ be a vector bundle of rank $r$ on $X$ whose restriction $E$ to $C$ has multi-degree $(w_1, w_2)$.
The integers $n_2 \geq n_1$ are such that we have vanishing $H^0(C_2, E_n|_{C_2}) = 0$ for all $n \leq n_1$ and $H^0(C_1, E_n|_{C_1}) = 0$ for $n \geq n_2$.
For convenience, we decrease $n_1$ and increase $n_2$ so that the vanishing on $C_2$ holds for all $n \leq n_1 - (w_1-2g_1)$ and on $C_1$ for all $n \geq n_2 + (w_2-2g_2)$.
Define
\[ d_1 = w_1 - n_1r, \quad d_2 = w_2 + n_2r,\text{ and } b = n_2 - n_1,\]
as before.

Set $\mathcal E' = \mathcal E \otimes \det \mathcal E \otimes \omega_{X/B}$.
Then $\mathcal E'$ is a vector bundle of rank $r$ on $X$ whose restriction $E'$ to $C$ has multi-degree $(w_1', w_2')$ where
\[ w_1' = w_1 + r(w_1-2g_1+1) \text{ and } w'_2 = w_2 + r(w_2-2g_2+1).\]
We set
\[ n_1' = n_1(1+r) \text{ and } n_2' = n_2(1+r),\]
and observe that we have vanishings $H^0(C_2, E'_{n}|_{C_2}) = 0$ for $n \leq n_1'$ and $H^0(C_1, E'_{n}|_{C_1}) = 0$ for $n \geq n_2'$.
We also set
\[ b' = n_2' - n_1' = b(1+r).\]

Our next goal is to define a rational map
\begin{equation}\label{eq:Rtilde}
  \rho \from {\mathcal G}(r+1, \mathcal E) \dashrightarrow {\mathcal G}(1, \mathcal E')
\end{equation}
that extends the projection-ramification map
\[
  \rho \from \Gr(r+1, H^0(X_\eta, \mathcal E_\eta)) \dashrightarrow \Gr(1, H^0(X_\eta, \mathcal E'_\eta))
\]
on $X_\eta$.
For technical reasons, we define the map in \eqref{eq:Rtilde} only on the reduced scheme underlying ${\mathcal G}(r+1, \mathcal E)$.

Before defining the map, we identify three conditions on linked linear series on the central fiber that are required for the map to be defined.
To do this, consider a linked linear series $(V_n \mid n \in \Z)$ on $C$, and let $(W_1, W_2)$ be the associated EHT limit linear series namely $W_1 = V_{n_1}$ and $W_2 = V_{n_2}$ (see \autoref{prop:llseht}).
The first condition we want to impose is that $(W_1, W_2)$ be a refined EHT limit linear series; this is an open condition (see \cite[Proposition~4.1.5]{oss:14}).
The second condition we want to impose is that the vanishing sequence of $W_1 \subset H^0(C_1, E_{n_1}|_{C_1})$ at $p$ is of the form
\begin{equation}\label{eqn:llsvs}
  (\underbrace{a, \dots, a}_i, \underbrace{a+1, \dots, a+1}_{r+1-i})
\end{equation}
as in \eqref{eqn:specialvs}; imposing a particular vanishing sequence is again an open condition (see \cite[Proposition~4.2.5]{oss:14}).
Since $(W_1, W_2)$ is refined, it follows that the vanishing sequence of $W_2 \subset H^0(C_2, E_{n_2}|_{C_2})$ at $p$ is
\[ (\underbrace{b-a-1, \dots, b-a-1}_{r+1-i}, \underbrace{b-a, \dots, b-a}_{i}).\]
Recall from \autoref{sec:prnongeneric} that $W_1$ yields two vector spaces $V^0$ and $V^1$ in the fiber $E_{n_1}|_p$, which we may identify canonically (up to scaling) with the fiber $E|_p$.
Likewise, $W_2$ yields two analogous vector spaces, call them $\Lambda^0$ and $\Lambda^1$, in $E|_p$.
The gluing condition in the definition of EHT limit linear series (\autoref{def:eht}) and the definition of these vector spaces immediately shows that
\begin{equation}\label{eqn:vlambdaswitch}
  V^0 = \Lambda^1 \text{ and } V^1 = \Lambda^0.
\end{equation}
The third condition we want to impose is that these two vector spaces be transverse, namely $\dim (V^0 \cap V^1) = 1$.

Let $\mathcal U \subset {\mathcal G}(r+1, \mathcal E)$ be the complement of the union of the following closed sets:
\begin{enumerate}
\item the closure of the subset of $\Gr(r+1, H^0(X_\eta, \mathcal E_\eta))$ corresponding to $V \subset H^0(X_\eta, \mathcal E_\eta)$ for which the evaluation map $V\otimes\O_{X_\eta} \to \mathcal E_\eta$ has generic rank less than $r$.
\item the set of linked linear series $(V_n \mid n \in \Z)$ on $C$ such that the associated EHT limit linear series $(W_1, W_2)$ is not refined, or does not have the vanishing sequence as in \eqref{eqn:llsvs}, or does not satisfy the transversality condition $\dim (V^0 \cap V^1) = 1$.
\end{enumerate}
Give $\mathcal U$ the reduced scheme structure.

Let $S$ be a reduced $B$-scheme with a map to $\mathcal U$ given by the linked linear series $(V_n \mid n \in \Z)$.
On $X_S$, we have a diagram analogous to \eqref{eqn:differential_construction}, namely
\begin{equation}
  \label{eq:llspr}
  \begin{tikzcd}
    &\det \mathcal E_n^* \otimes \det V_n\ar{r}{j}\ar{d}{d} & V_n \otimes \O_{X_S}\ar{r}{e}\ar{d}{e} & \mathcal E_n\ar[equal]{d}&\\
    0\ar{r} & \Omega_{X_S/S} \otimes \mathcal E_n\ar{r} & P(\mathcal E_n)\ar{r} &\ar{r} \mathcal E_n \ar{r}& 0.
  \end{tikzcd}
\end{equation}
Here $P(\mathcal E_n)$ is the sheaf of principal parts of $\mathcal E_n$ relative to $X_S \to S$, and the bottom row is the natural exact sequence coming from its definition.
The top row is a complex, but it may not be exact.
The maps labeled $e$ are the evaluation maps.
The map $j$ is defined by the maximal minors of $e \from V_n \otimes \O_{X_S} \to \mathcal E_n$.
The map $d$ is the unique map induced by the other maps in the diagram.
By composing $d$ through the inclusion $\Omega_{X_S/S} \to \omega_{X_S/S}$, and doing some rearrangement, we obtain a map
\begin{equation}\label{eqn:Rn}
r_n \from \det V_n \to \mathcal \pi_*(\mathcal E_n \otimes \det \mathcal E_n \otimes \omega^*_{X_S/S}) = \pi_*(\mathcal E'_{(r+1)n}).
\end{equation}
Consider the two extremal sections, namely those corresponding to $n = n_1$ and $n = n_2$.
\begin{lemma}\label{lem:rameht}
  Over every $s \in S$ over $0 \in \Delta$, the restrictions $r_{n_1}|_s$ and $r_{n_2}|_s$ define a one-dimensional refined EHT limit linear series for $E'$.
\end{lemma}
\begin{proof}
  Without further comment, we identify $r_{n_1}|_s \in H^0(C, E'_{(r+1)n_1})$ with its image in $H^0(C_1, E'_{(r+1)n_1}|_{C_1})$.
  We have
  \[E'_{(r+1)n_1}|_{C_1} = E_{n_1} \otimes \det E_{n_1} \otimes \omega_C|_{C_1} = E_{n_1} \otimes \det E_{n_1} \otimes \Omega_C|_{C_1} \otimes \O_{C_1}(p),\]
  and by construction $r_{n_1}|_s$ is the image of the ramification section of $V_{n_1} \subset H^0(C_1, E_{n_1}|_{C_1})$ under the inclusion map
  \[ E_{n_1} \otimes \det E_{n_1} \otimes \Omega_C|_{C_1} \to E_{n_1} \otimes \det E_{n_1} \otimes \omega_C|_{C_1} = E'_{(r+1)n_1}|_{C_1}.\]
  By \autoref{prop:agreement}, the ramification section of $V_{n_1}$ has order of vanishing $(r+1)a+(r-i)$ at $p$, and hence $r_{n_1}|_s$ on $C_1$ has order of vanishing $(r+1)a+(r-i+1)$ at $p$.
  Likewise, $r_{n_2}|_s$ on $C_2$ has order of vanishing $(r+1)(b-a-1)+i$ at $p$.
  Since
  \[ (r+1)a+(r-i+1) + (r+1)(b-a-1) + i = (r+1)b = b',\]
  we see that $r_{n_1}|_s$ and $r_{n_2}|_s$ have complementary orders of vanishing, leading to an equality in condition~\eqref{ieq:eht} of \autoref{def:eht}.

  We must next ensure that condition~\eqref{gluing:eht} of \autoref{def:eht} holds, that is, the images of $r_{n_i}|_s$ in the appropriate twists of $E_{n_i}|_p$ are equal, at least up to scaling.
  By \autoref{prop:agreement}, the image of $r_{n_1}|_s$ in the appropriate twist of $E_{n_1}|_p$ spans the line $(V^0 \cap V^1)$, and the image of $r_{n_2}|_s$ spans the line $\Lambda^0 \cap \Lambda^1$.
  But by \eqref{eqn:vlambdaswitch}, we have $V^1 = \Lambda^0$ and $V^0 = \Lambda^1$, so the two lines are equal.
\end{proof}

Thanks to \autoref{lem:rameht}, we apply \autoref{prop:llseht}, and conclude that there exists a unique (1-dimensional) linked linear series $(R_n \mid n \in \Z)$ of $\mathcal E'$ on $X_S$ for which $R_{n_1'} = \det V_{n_1}$ and $R_{n_2'} = \det V_{n_2}$, at least if $S$ is reduced.
The transformation
\[ (V_n \mid n \in Z) \mapsto (R_n \mid n \in \Z)\]
defines a morphism
\begin{equation}\label{prop:mapreduced}
  \rho \from \mathcal U \to \mathcal G(1, \mathcal E'),
\end{equation}
as desired in \eqref{eq:Rtilde}.
Note that $\mathcal U$ has the reduced scheme structure.

The fruit of our labor is the following corollary.
Let $\mathcal U_0$ be the fiber over $0$ of $\mathcal U \to B$.
\begin{corollary}\label{prop:degeneration}
  Suppose $v \in \mathcal U_0$ is such that $\dim_v \mathcal U_0 = (r+1)(d-rg-1)$ and $v$ is isolated in the fiber of $\rho$, then the projection-ramification map $\Gr(r+1, H^0(X_\eta, \mathcal E_\eta)) \dashrightarrow \P H^0(X_\eta, \mathcal E_\eta \otimes \det E_\eta \otimes K_{X_\eta})$ is generically finite.
\end{corollary}
\begin{proof}
  If $\dim_v \mathcal U_0 = (r+1)(d-rg-1)$, then $v$ is in the closure of $\Gr(r+1, H^0(X_\eta, \mathcal E_\eta))$ by \autoref{thm:lls}.
  The statement now follows from the upper semi-continuity of fiber dimension.
\end{proof}

\subsection{Maximal variation for generic scrolls of high degree}\label{sec:llsproof}
We now have all the tools to prove \autoref{thm:rationalnormalscrolls}.
\begin{theorem}[\autoref{thm:rationalnormalscrolls}]
  \label{thm:actualrationalnormalscrolls}
  Let $E$ be a generic vector bundle on $\P^1$ of rank $r$ and degree $d = a(r-1) + b(2r-1)+1$, where $a, b$ are positive integers.
  Then the projection-ramification map is generically finite, and hence dominant, for $E$.
  In particular, the projection-ramification map is dominant for generic $E$ of degree $\geq (r-1)(2r-1)+1$.
\end{theorem}
\begin{proof}
  We say that generic dominance holds for rank $r$ and degree $d$ if the projection-ramification map is dominant (equivalently, generically finite) for the generic vector bundle of rank $r$ and degree $d$.
  The rank will be fixed throughout, so let us drop it from the discussion.
  Let us prove that if generic dominance holds for degrees $d_1$ and $d_2$, then it also holds for degree $d = d_1 + d_2 - 1 $.
  With the base cases $d_1 = r$ (\autoref{prop:segre}) and $d_2 = 2r$ (\autoref{prop:222}), this proves the theorem.

  Take $C_1 = C_2 = \P^1$, and let $C = C_1 \cup C_2$ be their nodal union at one point, which we take to be the point labeled $0$ on both $\P^1$s.
  Let $X \to B$ be a smoothing of $C$.
  Note that any vector bundle on $C$ is the restriction of a vector bundle on $X$.
  Therefore, by \autoref{prop:degeneration}, it suffices to construct a vector bundle $E$ of degree $d$ on $C$ and a linked linear series $(V_n \mid n \in \Z)$ on $E$ such that the following conditions hold for the point $v$ of $\mathcal G (r+1, E')$ represented by $(V_n \mid n \in \Z)$:
  \begin{enumerate}
  \item $\dim_v \mathcal G(r+1, E) = (r+1)(d-1)$,
  \item $\rho$ is defined at $v$, and
  \item $v$ is an isolated point in the fiber of $\rho$.
  \end{enumerate}

  We construct $E$ as follows.
  Let $E_1$ be a generic vector bundle of degree $d_1$ on $C_1$, and $E_2'$ a generic vector bundle of degree $d_2 - 1$ on $C_2$.
  Choose a generic isomorphism $E_1|_0 \cong E_2'|_0$, and construct the vector bundle $E$ on $C$ by gluing $E_1$ and $E_2'$ along this isomorphism.
  Choose $n_1 = a$ and $n_2 = b+a$ for sufficiently negative $a$ and sufficiently positive $b$.
  The isomorphism $E_1 |_0 \cong E_2'|_0$ yields isomorphisms, canonical up to scaling, of $E_1(m)|_0$ and $E_2'(n)|_0$ for any $m, n \in \Z$.

  Having constructed $E$, we must now construct $(V_n \mid n \in \Z)$.
  By \autoref{prop:llseht}, it is enough to construct $V_{n_1} \subset H^0(C_1, E_1 \otimes \O(a))$ and $V_{n_2} \subset H^0(C_2, E_2'(b-a))$, provided they define a refined EHT limit linear series.
  Let $V \subset H^0(C_1, E_1)$ be a generic $(r+1)$-dimensional vector space.
  Then it will have the vanishing sequence $(0, \dots, 0, 1)$.
  Hence, we have $V^0 = E|_0$ and $V^1 \subset E|_0$ is $1$-dimensional (see \autoref{sec:prnongeneric} for the definition of these two subspaces).
  Furthermore, the genericity of $V$ implies that $V^1$ is a general $1$-dimensional subspace.
  Define $E_2$ by the sequence
  \[ 0 \to E_2 \to E_2'(1) \to E'_2(1)|_0 / V^1 \to 0.\]
  Let $\Lambda \subset H^0(C_2, E_2'(1))$ be the image of a general $(r+1)$ dimensional subspace of $H^0(C_2, E_2)$.
  Then $\Lambda \subset H^0(C_2, E_2'(1))$ has the vanishing sequence $(0, 1, \dots, 1)$, with $\Lambda^0 = V^1$ and $\Lambda^1 = V^0$.
  Let $V_{n_1} \subset H^0(C_1, E_1 \otimes \O(a))$ be the image of $V$ and $V_{n_2} \subset H^0(C_2,E'_2(b-a))$ the image of $\Lambda$.
  Then $V_{n_1}$ has the vanishing sequence $(a, \dots, a, a+1)$, and $\Lambda$ the complementary vanishing sequence $(b-a-1, b-a, \dots, b-a)$.
  By the construction of $\Lambda$, there exist bases of $V_{n_1}$ and $V_{n_2}$ that satisfy the gluing condition at $0$.
  In conclusion, $V_{n_1}$ and $V_{n_2}$ form a refined EHT limit linear series, and hence define a linked linear series $v = (V_n \mid n \in \Z)$.

  It is easy to check that $\dim_v \mathcal G(r+1, E) = (r+1)(d-1)$.
  Indeed, for every linked linear series $w = (W_n \mid n \in \Z)$ in an open subset around $v$, the EHT limit linear series associated to $w$ determines $w$ and has the same vanishing sequence as $v$.
  In particular, $W_{n_1} \subset H^0(C_1, E_1(a))$ is the image of an $(r+1)$-dimensional subspace $V(w) \subset H^0(C_1, E_1)$ with vanishing sequence $(0, \dots, 0, 1)$, and $W_{n_2} \subset H^0(C_2, E_2'(b-a))$ is the image of an $(r+1)$-dimensional subspace $\Lambda(w)$ of $H^0(C_2, E'_2(1))$ with vanishing sequence $(0,1,\dots,1)$.
  The gluing condition, in turn, implies that $\Lambda(w)$ is the image of an $(r+1)$-dimensional subspace of the kernel of the map
  \[ E_2'(1) \to E_2'(1)/V(w)^1.\]
  By the genericity of $V$, the isomorphism type of the kernel of this map is constant around $v$; that is, the kernel is isomorphic to $E_2$.
  So, a dimension count for $\mathcal G(r+1, E)$ around $v$ gives
  \begin{align*}
    \dim_v \mathcal G(r+1, E) &= \dim \Gr(r+1, H^0(C_1, E_1)) + \dim \Gr(r+1, H^0(C_2, E_2))\\
                              &= (r+1)(d_1-1) + (r+1)(d_2-1)\\
                              &= (r+1)(d_1+d_2-2)\\
                              &= (r+1)(d-1).\\
  \end{align*}

  Finally, we must check that $v$ is an isolated point in the fiber of
  \[ \rho \from \mathcal G(r+1, E) \dashrightarrow \mathcal G(1, E \otimes \det E \otimes \omega_C).\]
  For any $w \in \mathcal G(r+1, E)$ in an open set around $v$ with $w \neq v$, either $V(w) \neq V$ or $\Lambda(w) \neq \Lambda$, where $V, \Lambda, V(w), \Lambda(w)$ are as above.
  By construction, $V \subset H^0(r+1, H^0(C_1, E_1))$ and $\Lambda \subset H^0(r+1, H^0(C_2, E'_2(1)))$ are isolated in their respective projection-ramification maps.
  Therefore, either $\rho_{C_1} (V(w)) \neq \rho_{C_1}(V)$ or $\rho_{C_2}(\Lambda(w)) \neq \rho_{C_2}(\Lambda)$.
  In either case, we obtain that $\rho(v) \neq \rho(w)$, and hence conclude that $v$ is an isolated point in the fiber of $\rho$.
\end{proof}

\section{The Projection-Ramification enumerative problem} \label{sec:enumerativeproblems}
In this section, we calculate the degree of the projection-ramification map for as many varieties of minimal degree as we can, leading to a proof of \autoref{thm:examples}.
After treating the relatively easy cases by hand, we relate the projection-ramification map for the Veronese surface and the quartic normal scroll with classical geometry of cubic plane curves.

\subsection{Rational normal curves}
\label{sec:arnc}
Let $X \subset \P^n$ be a rational normal curve.
Plainly, $X$ is incompressible, and hence the projection-ramification map
\[ \rho \from \Gr(2, n+1) \to \P^{2n-2}\]
is a regular map.
Therefore, we get
\begin{align*}
  \deg \rho &= c_1(\rho^* \O(1))^{2n-2}\\
            &= c_1(\O_{\Gr(r+1, n+1)}(1))^{2n-2} \\
            &= \frac{(2n-2)!}{n!(n-1)!}.
\end{align*}

\subsection{Quadric hypersurfaces} \label{sec:aquadricsurface}
A smooth quadric hypersurface $X \subset \P^{n}$ defined by a homogeneous quadric equation $F(X_{0},\dots,X_n) = 0$.
An easy calculation shows that the projection-ramification map
\[ \rho \from \P^n \to (\P^n)^*\]
is given in coordinates by
\[ p = [p_0: \dots:p_n] \mapsto \left[ \frac{\partial F}{\partial X_0}(p): \dots: \frac{\partial F}{\partial X_n}(p) \right].\]
In other words, it is the \emph{polarity isomorphism} induced by $F$, namely the isomorphism between a projective space and its dual given by the non-degenerate bilinear form associated to $F$.
In particular, we get $\deg \rho = 1$.

\subsection{The Veronese surface} \label{sec:veronese}
Let $\P^{2} \cong X \subset \P^{5}$ be the Veronese surface, the image of $\P^2$ under the complete linear series $\O(2)$.
In this case, the projection-ramification map
\begin{align*}
  \rho \from  \Gr(3,H^0(\P^2, \O(2))) \cong \Gr(3, 6) \dashrightarrow \P H^0(\P^2, \O(3))^* \cong \P^9
\end{align*}
can be described as follows.
Let $N \subset H^0(\P^2, \O(2))$ be a net of conics.
Then $\rho(N)$ corresponds to the cubic curve traced out by the nodes of the singular members of $N$, called the \emph{Jacobian} of $N$.
\begin{proposition}\label{prop:veronese}
  Let $R \subset \P^2$ be a general cubic.
  The fiber of $\rho$ over $R$ is in natural bijection with the set of non-trivial $2$-torsion line bundles on $R$.
  In particular, we have $\deg \rho = 3$.
\end{proposition}
The rest of \autoref{sec:veronese} is devoted to the proof of this assertion.

For the proof, we recall some classical projective geometry of cubics and nets of conics from \cite[\S~3]{dol:12}.
To distinguish the various copies of $\P^2$ that naturally arise in this story, write $\P^2 = \P V$ for a 3 dimensional vector space $V$.
Let $N \subset H^0(\P V, \O(2)) = \Sym^2V$ be a general net of conics on $\P V$.
Given a point $x \in \P N^*$, we denote the associated conic by $Q_x$.

Associated to the net $N$ are three important cubic plane curves, namely the Jacobian curve, the discriminant curve, and the Hermite curve.
We have already seen the Jacobian curve $R \subset \P V$.
The \emph{discriminant curve} $D \subset \P N^*$ is the locus of $x \in \P N^*$ such that $Q_x$ is singular.
Since a pencil of conics contains three singular members, we see that $D$ is a cubic curve.
Note that if $Q_x$ is singular, then it is the union of two distinct lines in $\P V$.
A component line of $Q_x$ is called a \emph{Reye line}.
The \emph{Hermite curve} $E \subset \P V^*$ is the locus of Reye lines.
We leave it to the reader to check that it is a cubic curve.

The three cubic curves introduced above are inter-related.
First, we have an isomorphism $\tau \from D \to R$ defined by
\begin{equation}\label{eqn:DR}
  \tau \from x \mapsto \text{The singular point of $Q_x$}.
\end{equation}
Second, we have a degree 2 map $E \to D$ defined by
\[ \ell \mapsto \text{The $x \in D$ such that $Q_x$ contains $\ell$}.\]
Evidently, the fiber of this map over a given $x \in D$ corresponds to the two components of $Q_x$.
The (\'etale) degree 2 map $E \to D \cong R$ gives a non-trivial 2-torsion element $\eta \in \Pic(R)[2]$.
The element $\eta$ is characterized by the property that it is the unique non-trivial 2-torsion element whose pull-back to $E$ is trivial.

Denote by $H$ the hyperplane divisor class on $R \subset \P^2$.
\begin{lemma}\label{lem:reye}
  For every $a \in R$, the line joining $a$ and $a+\eta$ is a Reye line.
  Furthermore, this Reye line is a component of $Q_d$ where $d = \tau^{-1}(H-2a-\eta)$.
  Finally, the conjugate Reye line, namely the other component of $Q_d$, passes through the points $b$ and $b+\eta$ where $b \in R$ differs from $a$ by a non-trivial 2-torsion element other than $\eta$.
\end{lemma}
\begin{proof}
  Let $\ell$ be a general Reye line, and let $d \in D$ be such that $\ell$ is a component of $Q_d$.
  Let $x = \tau(d) \in R$ be the singular point of $Q_d$.
  Note $\ell \cap R$ consists of three points, one of which is $x$.
  It suffices to show that the other two, say $y$ and $z$, differ by $\eta$.

  The point $y$ defines a line in $\P V^*$.
  This line intersects $E \subset \P V^*$ in three points, one of which is $\ell$, and the other two are the two components of $Q_{\tau^{-1}(y)}$, namely the two pre-images of $y \in R$ under the double covering $E \to R$.
  Call these two points $y_1$ and $y_2$.
  Define $z_1$ and $z_2$ analogously.
  By construction, the triplets $y_1, y_2, \ell$ and $z_1, z_2, \ell$ are collinear triplets on $E \subset \P V^*$, and therefore we have the linear equivalence
  \[ y_1 + y_2 \sim z_1 + z_2\]
  on $E$.
  By pushing this forward to $R$, we get
  \[ 2y \sim 2z.\]
  Therefore, $y - z$ is a (non-trivial) 2-torsion element in $\Pic(R)$.
  However, the pull-back of $y-z$ is trivial on $E$, and hence $y - z = \eta$.

  Finally, let $m$ be the Reye line conjugate to $\ell$.
  Then it contains $x$, and two other points of $R$, say $y'$ and $z'$.
  By what we just proved, $y' - z' = \eta$.
  But we also have $y' + z' \sim y + z$.
  Hence $y - y'$ is a 2-torsion element, non-trivial, and distinct from $\eta$.
  The proof is now complete.  
\end{proof}

We now have all the tools to prove \autoref{prop:veronese}.
\begin{proof}[Proof of \autoref{prop:veronese}]
  Let $U \subset \P H^0(\P^2, \O(3))^*$ be the locus of smooth cubic curves, $J \to U$ be the universal Picard scheme, $J[2] \subset J$ the closed subscheme of 2-torsion classes, and $J[2]^* \subset J[2]$ the open and closed subscheme of non-trivial 2-torsion classes.
  The projection-ramification map for the Veronese surface factors as
  \begin{alignat*}{2}
    \rho \from \Gr(3, H^0(\P^2, \O(2))) &\dashrightarrow J[2]^* &&\dashrightarrow \P H^0(\P^2, \O(3))^*\\
    N &\mapsto (R, \eta) &&\mapsto R.
  \end{alignat*}
  We construct $J[2]^* \to \Gr(3,H^0(\P^2, \O(2)))$ inverse to the first map, which shows that $\deg \rho = 3$, and identifies the fibers of $\rho$ as non-trivial 2-torsion points.
  Given $(R, \eta) \in J[2]^*$, we need to construct a net $N$ of conics with Jacobian $R$.
  We use \autoref{lem:reye}, which tells us the singular elements of this net in terms of $R$ and $\eta$.
  Let $\{\eta, \eta', \eta''\}$ be the three non-trivial 2-torsion line bundles on $R$.
  Define the map $R \to \P H^0(\P^2, \O(2))^*$ by
  \[ R \ni a \mapsto \left(\langle {a, a+\eta} \rangle\right) \cdot \left(\langle {a+\eta', a+\eta''} \rangle \right),\]
  where $\langle {p,q} \rangle$ denotes the line joining $p$ and $q$.
  We leave it to the reader to check that the image of $R$ is a plane cubic curve.
  The span of the image of $R$ is the desired net $N$.  
\end{proof}

\subsection{Quartic surface scroll}\label{sec:quartic_scroll}
Our next objective is to prove that $ \deg \rho_{X} = 2$ for a generic quartic surface scroll $X \subset \P^{5}$.
We begin by recasting $\rho_X$ in terms of nets of conics on $\P^2$, and bring in the projective geometry introduced in \autoref{sec:veronese}.

The generic quartic surface scroll $X \subset \P^5$ is isomorphic to $\P^1 \times \P^1$, embedded by the complete linear system associated to $\O(1,2)$.
Say $\P^1 \times \P^1 = \P U \times \P V$, where $U$ and $V$ are two-dimensional vector spaces.
Then the projection-ramification map is a $\PGL(U) \times \PGL(V)$-equivariant map
\[
  \Gr(3, U \otimes \Sym^2 V) \dashrightarrow \P(U \otimes \Sym^4 V)^*.
\]
We take the quotient of both sides by the $\PGL(U) \times \PGL(V)$-action.
We begin by identifying the two quotients.

Let $S$ be a 3-dimensional quadratic space, that is, a vector space with a non-degenerate quadratic form $q$.
Then we have $\Aut(S) =  \operatorname{O}(q) \cong \operatorname{O(3)}$.
The projective space $\P S$ is isomorphic to $\P^2$, and it comes with a distinguished smooth conic $Q \subset \P S$.
The automorphism group of the pair $(\P S, Q)$ is $\Aut(Q) \cong \PGL_2$.
\begin{lemma}\label{lem:quotgrass}
  The quotient $\Gr(3, U \otimes \Sym^2V) / \PGL(U) \times \PGL(V)$ is birational to the quotient $\Hilb^3(\P S) / \Aut S$.
\end{lemma}
\begin{proof}
  Let $W$ be a 3-dimensional vector space.
  We have a birational isomorphism
  \begin{align*}
    \Gr(3, &U \otimes \Sym^3 V) / \PGL(U) \times \PGL(V)\\
           &\sim (W^* \otimes U \otimes \Sym^2 V) / \GL(W) \times \GL(U) \times \GL(V).
  \end{align*}
  Interpret the space $(W^* \otimes U \otimes \Sym^2 V) / \GL(W) \times \GL(U)$ as the space of $2 \times 3$ matrices with entries in $\Sym^2 V$, modulo row and column transformations.
  Set $S = \Sym^2 V$; it has a canonical (up to scaling) quadratic form given by the conic $Q \cong \P V\subset \P S$ embedded by $\O(2)$.
  We can then interpret $(W^* \otimes U \otimes \Sym^2 V) / \GL(W) \times \GL(U)$ as the space of $2 \times 3$ matrices with entries in $S$.
  We have a rational map
  \begin{align*}
    (W^* \otimes U \otimes \Sym^2 V) / \GL(W) \times \GL(U) &\sim \Hilb^3(\P S) \\
    \text{$2 \times 3$ matrix $M$} &\mapsto \text{Vanishing locus of $2\times 2$ minors of $M$}.
  \end{align*}
  It is easy to check that this map is a birational isomorphism---a general triple of points in $\P S$ is the zero locus of $2 \times 2$ minors of a matrix of linear forms, which is uniquely determined up to row and column transformations.
  By taking a further quotient by $\GL(V)$, we finish the proof.  
\end{proof}

\begin{lemma}\label{lem:quotram}
  The quotient $\P(U \otimes \Sym^4 V)^* / \PGL(U) \times \PGL(V)$ is birational to the quotient $\Gr(2, (\Sym^2 S)/q) / \Aut S$.
\end{lemma}
\begin{proof}
  We have the birational isomorphism
  \begin{align*}
    (U \otimes \Sym^4 V) / \GL(U) &\sim \Gr(2, \Sym^4 V).
  \end{align*}
  Note that $q \in \Sym^2S$ spans the kernel of the natural surjection $\pi \from \Sym^2 S \to \Sym^4 V$.
  So the claimed birational isomorphism is given by sending a two dimensional subspace $L \subset \Sym^4 V$ to the image of $\pi^{-1}L$ in $\Sym^2S/q$.
\end{proof}

Via the birational isomorphisms in \autoref{lem:quotgrass} and \autoref{lem:quotram}, the projection-ramification map $\mu$ transforms into an $\Aut(S)$-equivariant map
\[
  \mu \from \Hilb^3\P S \dashrightarrow \Gr(2, \Sym^2S/q).
\]
We now describe this map $\mu$.
To ease notation, we denote a linear form and its vanishing locus by the same letter.
Let $\xi \in \Hilb^3 \P S$ be a general point corresponding to the three vertices of the triangle formed by three lines $L_i$ for $i = 1, 2, 3$.
Two lines $L_i$ and $L_j$ define a pencil of quadratic forms on $Q$.
Let $R_{ij}$ be the line whose intersection with $Q$ is the ramification divisor of the pencil $\langle  L_i, L_j \rangle$.
It is easy to check that the quadrics $L_1 R_{23}$, $L_2 R_{13}$, and $L_3R_{12}$ span a 3-dimensional subspace of $\Sym^2 S$ that contains the quadric $q$.
\begin{lemma}\label{lem:mu}
  In the setup above, the image of $\xi$ under $\mu$ is the image of $\langle  L_1R_{23}, L_2R_{13},L_3R_{12} \rangle$ in $\Sym^2S/q$.
\end{lemma}
\begin{proof}
  The ideal of the point $\xi \in \Hilb^3(S)$ is cut out by $2 \times 3$ matrix of linear forms
  \[
    M =
    \begin{pmatrix}
      L_1 & 0 & L_3 \\
      0 & L_2 & L_3
    \end{pmatrix}.
  \]
  Let $U_0, U_1$ be a basis of $U$.
  Under the isomorphism in \autoref{lem:quotgrass}, this $2 \times 3$ matrix corresponds to the point of $\Gr(3, U \otimes \Sym^2 V)$ given by the subspace of $U \otimes \Sym^2 V$ spanned by $U_0 M_{0,i} + U_1 M_{1,i}$ for $i = 1, 2, 3$.
  From \eqref{eqn:Rmatrix}, the ramification divisor of this subspace is given by
  \begin{align*}
    R &= \det
    \begin{pmatrix}
      L_1 & 0 & U_0 L_1' \\
      0 & L_2 & U_1 L_2' \\
      L_3 & L_3 & (U_0+U_1)L_3'
    \end{pmatrix} \\
      &= U_0L_2(L_3'L_1 - L_1L_3') + U_1L_1(L_3'L_2-L_2L_3')\\
      &= U_0L_2R_{13} + U_1 L_1R_{23}.
  \end{align*}
  In this calculation, $L_i'$ denotes the derivative $\frac{d}{dt}$ of $L_i$ considered as an element of $\k[t]$ by pullback under some parametrization $\spec \k[t] \to Q$ and trivialization of $\O(2)|_{\spec \k[t]}$.
  Although the derivative depends on the choices, the forms $L_iL_j' - L_jL_i'$ do not, and they cut out precisely the ramification divisor of the pencil $\langle  L_i, L_j \rangle$.
  Under the isomorphism in \eqref{lem:quotram}, the divisor $R$ corresponds to the 2 dimensional subspace of $\Sym^2 S/q$ spanned by $L_2R_{13}$ and $L_1R_{23}$ (The roles of $L_1, L_2, L_3$ can be changed by linear transformations of $M$, so we get that $L_3R_{12}$ also lies in this span).
  The proof is thus complete.
\end{proof}

Recall that the conic $Q \subset \P S$ gives an isomorphism $\P S \cong \P S^*$, called \emph{polarity} with respect to $Q$.
On the vector spaces, it is the isomorphism induced by the bilinear form associated to $q$.
Geometrically, it is characterized by the rule that the polar of a point $p \in Q$ is the tangent line to $Q$ at $p$.
More generally, given a point $p \in \P S$, the pencil of lines through $p$ contains two lines tangent to $Q$; the polar of $p$ is the line joining the two points of tangency.
We denote the polar of a point $p$ (resp. a line $L$) by $p^\perp$ (resp. $L^\perp$).

Set $M_i = R_{jk}$, and let $N$ be the net spanned by $L_iM_i$ for $i = 1, 2, 3$.
By the definition of $R_{jk}$, we see that $M_i$ is the polar line of the point $L_j \cap L_k$.
In other words, the triangles $(L_1, L_2, L_3)$ and $(M_1, M_2, M_3)$ are polar conjugates---lines in one are polars to the vertices of the other.

\begin{remark}\label{rem:richelot}
  The space $\Hilb^3 \P S / \Aut S$ and its birational involution (called the Richelot involution) induced by
  \begin{equation}\label{eqn:richelot}
    \tau \from \Hilb^3 \P S \dashrightarrow \Hilb^3 \P S
  \end{equation}
  that sends a triangle formed by the lines $(\ell_1, \ell_2, \ell_3)$ to the triangle formed by the points $(\ell_1^\perp, \ell_2^\perp, \ell_3^\perp)$ have well-known moduli interpretations, which we learned from \cite[Example~4.2]{dol.how:15}.
  Intersecting the lines $\ell_i$ with $Q$ for $i = 1, 2, 3$ gives a triple of pairs of points on $Q \cong \P^1$.
  The double cover of $\P^1$ branched along these six points gives a genus 2 curve $C$.
  The grouping of the six points in three pairs $\{p_i, q_i\}$ for $i = 1, 2, 3$ gives a 2-dimensional subspace of the 4-dimensional $\F_2$-vector space $\Pic C [2]$, namely $\{0\} \cup \{p_i - q_i \mid i = 1,2,3\}$.
  This vector space is a maximal isotropic subspace for the Weil pairing on $\Pic C[2]$.
  Conversely, a genus 2 curve $C$ with a maximal isotropic subspace of $\Pic C[2]$ defines six points on $\P^1$ grouped into a triple of pairs.
  Thus, $\Hilb^3 \P S / \Aut S$ is birational to the moduli space of genus 2 curves along with an isotropic subspace of 2-torsion points in its Jacobian.

  Since general principally polarized abelian surfaces are Jacobians of genus 2 curves, $\Hilb^3 \P S / \Aut S$ is also birational to the moduli of $(A, G)$, where $A$ is a principally polarized abelian surface and $G \subset A[2]$ is a maximal isotropic subspace for the Weil pairing.
  In this interpretation, the involution $\tau$ is called the Fricke involution; it sends $(A, G)$ to $(A/G, A[2]/G)$ (see \cite[Page~2]{muk:12}).

  Using the Torelli theorem, the moduli space of $(A, G)$ can be described as the quotient $H_2/\Gamma_0(2)$, where $H_2$ is the Siegel upper half space of degree 2, and $\Gamma_0(2) \subset \operatorname{Sp}(4, \Z)$ is the congruence subgroup consisting of matrices $\begin{pmatrix} A & B \\ C & D \end{pmatrix}$ where $A, B, C, D$ are $2 \times 2$ blocks and $C \equiv 0 \pmod 2$.
  In this interpretation, the involution $\tau$ is induced by the action of $\frac{1}{\sqrt 2} \begin{pmatrix} 0 & I_2 \\ -2 I_2 & 0 \end{pmatrix} \in \operatorname{Sp}(4, \R)$ on $H_2$ (again, see \cite[Page~2]{muk:12}).

  Finally, suppose we pass to ordered triples of pairs, or equivalently to $(\P S)^3 / \Aut S$.
  Then the moduli interpretation changes slightly.
  Now the space is the moduli space of $(A, \psi)$, where $A$ is a principally polarized abelian surface and $\psi \from \F_2^2 \to A[2]$ is an isomorphism onto a maximal isotropic subspace (note that $\psi$ is equivalent to a maximal isotropic $G \subset A[2]$ along with a basis of $G$).
  The moduli space of $(A, \psi)$ is the quotient $H_2 / \Gamma_1(2)$ where $\Gamma_1(2) \subset \operatorname{Sp}(4, \Z)$ is defined by the congruence conditions $A - I_2 \equiv 0 \pmod 2$ in addition to $C \equiv 0 \pmod 2$.
  The Fricke involution continues to act on $H_2 / \Gamma_1(2)$.
  The Satake compactification of $H_2 / \Gamma_1(2)$ is the Igusa quartic threefold in $\P^4$ on which the Fricke involution acts by a linear transformation of $\P^4$.
  The quotient of the Igusa quartic by the Fricke involution is isomorphic to a double cover of $\P^3$ branched along the union of 4 planes  \cite[Theorem~2]{muk:12}.
\end{remark}

Recall that $\xi \in \Hilb^3\P S$ is the point defined by the three vertices of the triangle formed by $(L_1, L_2, L_3)$.
Let $\xi' \in \Hilb^3\P S$ be the point defined by the three vertices of the triangle formed by $(M_1, M_2, M_3)$.
\begin{proposition}\label{prop:quarticscroll}
  In the setup above, $\xi$ and $\xi'$ are the only points of $\Hilb^3\P S$ that map to $N \in \Gr(2, \Sym^2S/q)$.
  In particular, the degree of $\mu \from \Hilb^3\P S \dashrightarrow \Gr(2, \Sym^2S/q)$ is 2.
\end{proposition}
We invite the reader to look at \autoref{fig:quarticscroll} for the configuration formed by the conic $Q$, the Jacobian $R$, the dual of the Hermite curve $E^\perp$, and the triangles $(L_1,L_2,L_3)$ and $(M_1,M_2,M_3)$.
\begin{proof}
  By \autoref{lem:mu}, we see immediately that $\mu(\xi') = \mu(\xi) = N$.
  To show that no other triangles map to $N$, consider pairs of triplets $\Delta = (\Delta_1, \Delta_2, \Delta_3)$ and $\nabla = (\nabla_1, \nabla_2, \nabla_3)$ of lines in $\P S$ such that
  \begin{enumerate}
  \item $\Delta$ and $\nabla$ are polar conjugates with respect to $Q$, and
  \item $\Delta_i \cup \nabla_i$ is an element of $N$ for $i = 1, 2, 3$.
  \end{enumerate}
  It suffices to show that the only ones satisfying the two conditions are $(L_1, L_2, L_3)$ and $(M_1, M_2, M_3)$, up to permutation.
  
  To show this, we need some observations.

  First, suppose $A_1 \cup B_1$ and $A_2\cup B_2$ are elements of the net $N$, where $A_i$ and $B_j$ are lines in $\P S$.
  Then, by definition, $A_i$ and $B_j$ are Reye lines of the net $N$.
  Let $p = A_1 \cap A_2$ and $t = B_1 \cap B_2$.
  We claim that the third Reye line through $p$, in addition to $A_1$ and $A_2$, is the line $\langle p, t\rangle$.
  Indeed, in the pencil of conics spanned by $A_1 \cup B_1$ and $A_2 \cup B_2$, the third singular conic is $\langle p, t\rangle \cup \langle p', t'\rangle$, where $p' = A_1 \cap B_2$ and $t' = A_2 \cap B_1$.

  Second, let $R \subset \P S$ be the Jacobian cubic and $E \subset \P S^*$ be the Hermite cubic of $N$.
  Let $E^\perp \subset \P S$ be the image of $E$ under the polarity isomorphism $\P S^* \to \P S$ induced by $Q$.
  Explicitly, the points of $E^\perp$ are the polars of the Reye lines.
  We claim that the six points of intersection of $R$ and $Q$ also lie on $E^\perp$.
  Indeed, to show that $x \in R \cap Q$ also lies on $E^\perp$, it suffices to show that the line $T_xQ$ is a Reye line.
  Since $x \in R$, there exists an element of $N$ of the form $A\cup B$ where $A$ and $B$ are lines intersecting at $x$.
  Note that in the pencil of conics spanned by $A \cup B$ and $Q$, there is a singular conic containing $T_xQ$.
  Therefore, $T_xQ$ is a Reye line.
  
  Third, since $R \cap E^\perp$ contains 6 points on the conic $Q$, the residual 3 points are collinear.
  Let them correspond to $x_1, x_2, x_3 \in E$.
  Denoting by $H$ the hyperplane class of $E \subset \P S^*$, we have the equation in $\Pic E$
  \[ x_1 + x_2 + x_3 = H.\]
  
  Suppose we have two triangles $\Delta$ and $\nabla$ satisfying the two conditions above.
  Consider the point $p_3 = \Delta_3 \cap \nabla_3$.
  By the second condition, it lies on $R$.
  By the polar conjugacy of $\Delta$ and $\nabla$, we have
  \begin{align*}
    p_3^\perp &= \langle \Delta_3^\perp, \nabla_3^\perp \rangle \\
              &= \langle  \nabla_1 \cap \nabla_2, \Delta_1 \cap \Delta_2\rangle.
  \end{align*}
  By the first observation, we see that $p_3^\perp$ is a Reye line.
  Hence $p_3$ lies on $E^\perp$, and hence on $R \cap E^\perp$.
  Similarly, $p_1 = \Delta_1 \cap \nabla_1$ and $p_2 = \Delta_2 \cap \nabla_2$ also lie on $E^\perp$.
  Since $N$ is general, we may assume that the $p_i$ do not lie on $Q$.
  Hence, $p_1, p_2, p_3$ are the three collinear points in $R \cap E^\perp$.
  (The fact that $p_1, p_2, p_3$ are collinear is not surprising---it is because any two polar conjugate triangles are in linear perspective \cite[Theorem~2.1.9]{dol:12}).
  By reordering if necessary, assume that we have $p_i^\perp = x_i$ as elements of $E$.

  Now, observe that the three Reye lines through the vertex $\Delta_1 \cap \Delta_2$ are $\Delta_1$, $\Delta_2$, and $p_3^\perp$, and likewise for the other two vertices.
  The concurrence of the three lines, along with the equality $p_3^\perp = x_1$, yields the system of equations on $\Pic E$
  \begin{align*}
    \Delta_1 + \Delta_2 + x_3 &= H, \\
    \Delta_2 + \Delta_3 + x_1 &= H, \\
    \Delta_3 + \Delta_1 + x_2 &= H. \\
  \end{align*}
  Of course, the same three equations hold if we replace $\Delta$ by $\nabla$.

  Note that the points $x_1, x_2, x_3 \in E$ are determined by $N$.
  Using $x_1 + x_2 + x_3 = H$, a simple calculation gives $2 \Delta_1 = 2x_1$.
  This equation has 4 solutions for $\Delta_1$, namely $x_1 + \epsilon$ for $\epsilon \in \Pic E[2]$.
  Also, $\Delta_1$ determines $\Delta_2$ and $\Delta_3$ by the equations above, which in turn determine the $\nabla_i$ using polarity or the property that $\nabla_i$ and $\Delta_i$ form a fiber of the map $E \to R$.
  Thus, it suffices to show that at most two of the four solutions for $\Delta_1$ can be valid.

  Suppose $\Delta_1 = x_1$.
  Then we get $\Delta_2 = x_2$, and $\Delta_3 = x_3$.
  However, the lines represented by the $x_i$ are concurrent, whereas the lines $\Delta_i$ are not.
  Therefore, we get that $\Delta_1 \neq x_1$.
  The same argument shows that $\nabla_1 \neq x_1$.
  Let the involution of $E$ induced by $E \to R$ be given by the addition of $\epsilon_0 \in \Pic E [2]^*$.
  Since $\Delta_1$ and $\nabla_1$ form a fiber of $E \to R$, we have $\nabla_1 = \Delta_1 + \epsilon_0$.
  So, $\nabla_1 \neq x_1$ translates into $\Delta_1 \neq x_1 + \epsilon_0$.
  In summary, the only two possible solutions for $\Delta_1$ are $x_1 + \epsilon$ for $\epsilon \in \Pic E [2] \setminus \{0, \epsilon_0\}$.
  The proof is now complete.  
\end{proof}
\vfill
\begin{figure}
  \centering
  \includegraphics{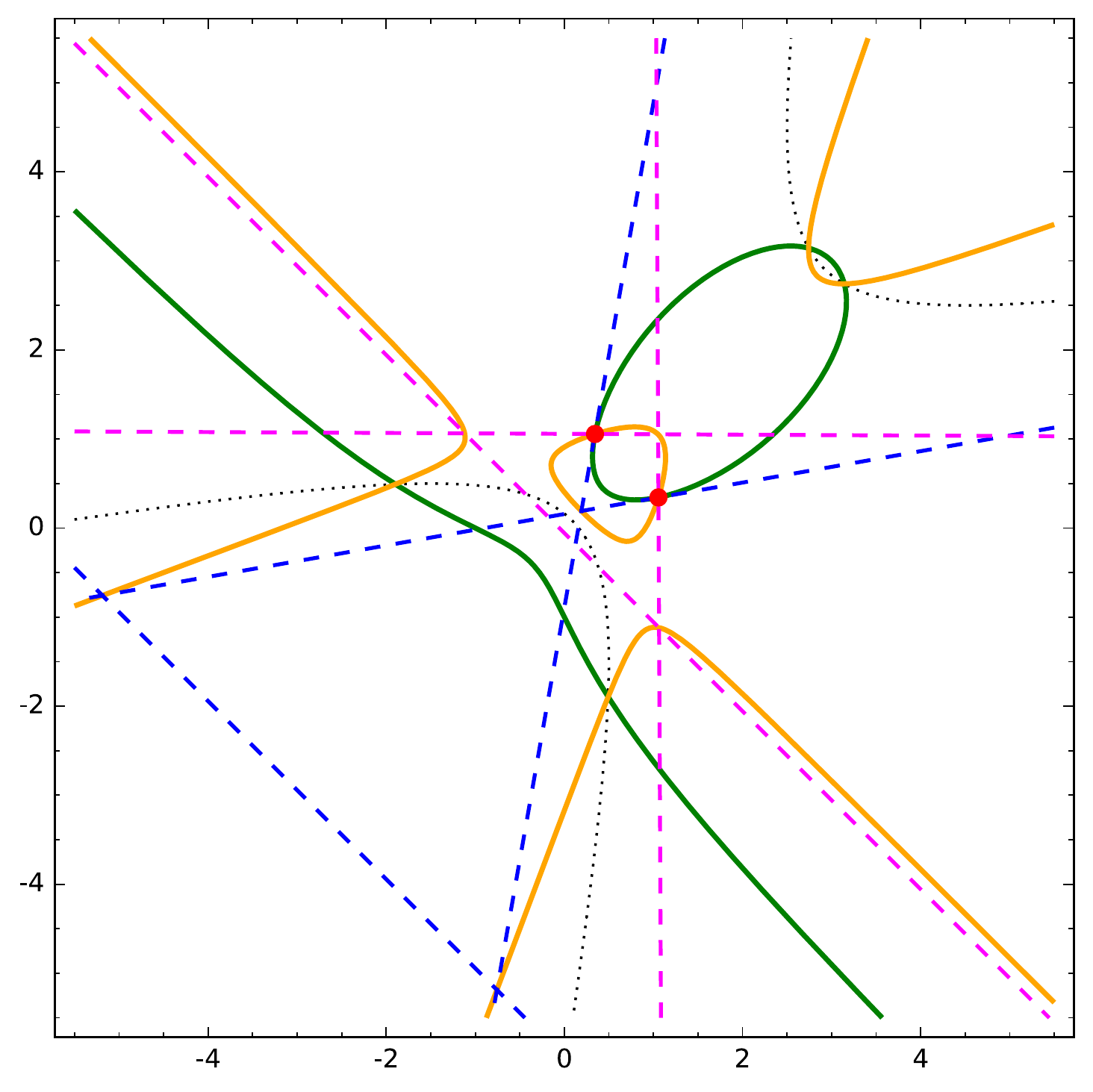}
    \caption{
    The figure shows the unique pair of polar conjugate triangles of Reye lines associated to a net of conics as proved in \autoref{prop:quarticscroll}.
    The net includes the distinguished conic $Q$, seen here as a hyperbola (dotted black).
    The Jacobian cubic $R$ (green) and the dual of the Hermite cubic $E^\perp$ (orange) intersect in 6 points on $Q$ (4 of which are real and visible), and 3 other collinear points, two of which, say $p_1$ and $p_2$, are marked (red), and the third, say $p_3$, is at infinity.
    The Reye lines forming the two polar conjugate triangles (dashed pink and dashed blue) come in conjugate pairs.
    The lines in each pair intersect at $p_1$, $p_2$, and $p_3$.
    (The figure was produced using \texttt{Sage} \cite{the:17}.)
  }
  \label{fig:quarticscroll}
\end{figure}

 \bibliographystyle{siam}

\end{document}